\documentclass[12pt,a4paper, reqno]{amsart}
\usepackage{}
\usepackage{txfonts}
\usepackage{color}
\usepackage{amsmath,amssymb,extarrows}
\usepackage{tikz,enumerate}
\usepackage{diagbox}
\usepackage{appendix}
\usepackage{epic}

\usepackage{float}
\usepackage{fourier}
\usepackage{bbm}
\usepackage{amsfonts}
\usepackage{amsxtra}
\usepackage{amssymb}
\usepackage{amscd}
\usepackage{amsthm}
\usepackage{mathrsfs}
\usepackage{hyperref}
\usepackage{leqno}
\usepackage{multirow}
\usepackage{tikz}
\usepackage{tikz-cd}

\makeatletter

\numberwithin{equation}{section}
\newtheorem*{theorem*}{Theorem}
\newtheorem{lemma}{Lemma}[section]
\newtheorem{proposition}[lemma]{Proposition}
\newtheorem{remark}[lemma]{Remark}

\newtheorem{theorem}[lemma]{Theorem}

\newtheorem{corollary}[lemma]{Corollary}
\newtheorem{question}[lemma]{Question}

\newtheorem*{question*}{Question}
\newtheorem*{assumption*}{Assumption}
\newtheorem*{axiom*}{Axiom}

\newtheorem*{theorem*1}{Theorem (\ref{theta1})}
\newtheorem*{theorem*2}{Theorem (\ref{theta2})}
\newtheorem*{theorem*3}{Theorem (\ref{theta3})}
\newtheorem*{theorem*4}{Theorem (\ref{theta4})}
\newtheorem*{proposition*5}{Proposition (\ref{theta5})}
\newtheorem*{proposition*6}{Proposition (\ref{theta6})}
\baselineskip 20pt \textwidth 18cm  \theoremstyle{plain}

\newcommand{\Aut}{\operatorname{Aut}}
\newcommand{\End}{\operatorname{End}}
\newcommand{\Hom}{\operatorname{Hom}}

\newcommand{\cInd}{\operatorname{c-Ind}}
\newcommand{\Ind}{\operatorname{Ind}}

\newcommand{\Res}{\operatorname{Res}}

\newcommand{\Ha}{\operatorname{H}}

\newcommand{\C}{\mathbb C}
\newcommand{\F}{\mathbb F}

\newcommand{\Q}{\mathbb Q}

\newcommand{\Z}{\mathbb Z}

\newcommand{\GL}{\operatorname{GL}}

\newcommand{\GSp}{\operatorname{GSp}}
\newcommand{\PGSp}{\operatorname{PGSp}}
\newcommand{\GPs}{\operatorname{GPs}}

\newcommand{\Sp}{\operatorname{Sp}}
\newcommand{\Ps}{\operatorname{Ps}}

\newcommand{\Span}{\operatorname{Span}}

\newcommand{\supp}{\operatorname{supp}}

\begin{document}
\title{Extended  Weil representations: the finite field cases}
\author{Chun-Hui Wang}
\address{School of Mathematics and Statistics\\Wuhan University \\Wuhan, 430072,
P.R. CHINA}
\email{cwang2014@whu.edu.cn}
\begin{abstract}
It is well known(cf. Weil, G\'erardin's works) that  there are  two different Weil representations of a symplectic  group  over an odd finite field. By a twisted action,  we show that one can reorganize them    as a  representation of  a related  projective symplectic similitude  group. We also discuss the even  field case by following Genestier-Lysenko and  Gurevich-Hadani's works on geometric Weil representations in characteristic two. As a result, we approach some of their results from the lattice model, which is inspired by MVW, Prasad and Takeda's works.
\end{abstract}
\maketitle
\setcounter{secnumdepth}{3}
\tableofcontents{}
\section{Introduction}
Let $F$   firstly be   a finite field of  odd cardinality $q$, and let $(W, \langle, \rangle)$  be a symplectic vector space over $F$ of dimension $2m$. The Heisenberg group $\Ha(W)$, attached to $W$ and $F$, is a set $W\oplus F$ with the group law: $(w, t) (w', t')=(w+w', t+t' + \tfrac{\langle w, w'\rangle}{2}).$
Let $\Sp(W)$, $\GSp(W)$ be the isometric group, resp. isometric similitude group of $(W, \langle, \rangle)$. Let $\psi$ be  a non-trivial character of $F$.   According to the Stone-von Neumann theorem, there is only one equivalence class of irreducible complex representation $\pi_{\psi}$ of $\Ha(W)$ with central character $\psi$.
By Weil's celebrated paper \cite{We}, $\pi_{\psi}$ is a representation of $\Sp(W)\ltimes \Ha(W)$. The restriction of $\pi_{\psi}$ to $\Sp(W)$, now is well-known as the \emph{Weil representation};  G\'erardin investigated fully this representation in \cite{Ge}.  Following Shinoda \cite{Sh},   one can  extend it to the  symplectic similitude group $\GSp(W)$ simply by setting $\rho=\Ind_{\Sp(W)}^{\GSp(W)} \pi_{\psi}$. This is a big representation. Inspired  by \cite{Ba} and \cite{GePi}, we  show that  it is also possible to extend  $\pi_{\psi}$ to  a related projective symplectic similitude   group  by some twisted actions.  More precisely,  let  $a\in F^{\times}$, and  let $\psi^a$ denote another character of $F$ defined as $t\longmapsto \psi(at)$, for $t\in F$. Let $\pi_{\psi^a}$ be the Weil representation of $\Sp(W)$ associated to $\psi^a$. By \cite[p.36(4)]{MoViWa}, $\pi_{\psi^{at^{2}}}\simeq \pi_{\psi^a}$, for any $t\in F^{\times}$. Therefore, we define a quotient group $\PGSp^{\pm}(W)$ from $\GSp(W)$ such that  the following  exact sequence holds:
\begin{align*}
1 \longrightarrow \Sp(W)\longrightarrow \PGSp^{\pm}(W)\longrightarrow F^{\times}/F^{\times 2}\longrightarrow 1.
\end{align*}
We define such a group case by case according to  $q \equiv 1( \mod 4)$ or   $q \equiv 3( \mod 4)$.  Details can be found in the section \ref{oddcase}. As a consequence, we can define the twisted induced Weil representation of $\PGSp^{\pm}(W)$: $\rho_{\psi}=\Ind_{\Sp(W)}^{\PGSp^{\pm}(W)} \pi_{\psi}$, which consists of  the two  different Weil representations of  $\Sp(W)$. We also consider whether such representation can arise from the Heisenberg representation of $\Ha (W)$.

Let $F$   secondly be  a finite field of  characteristic $2$.  In this case, the situation is not the same as in the above odd field setting.  Originally, Weil  considered his representation as a representation of   a pseudosymplectic group. For generalizations  of  this part, one can see \cite{Bl} and \cite{Ta}. Later, in \cite{GeLy}, \cite{GuHa},  Genestier-Lysenko and Gurevich-Hadani  demonstrated that  the Weil representation can also be considered as a representation of a Metaplectic group, which is a $4$-covering over an affine symplectic group. It should be noted that an affine symplectic group is a  bigger group than the corresponding pseudosymplectic group.  In this paper, we will   follow the second point of view. However, Genestier-Lysenko and Gurevich-Hadani obtained their results mainly using  the algebraic-geometric methods. It's much more complicated and sophisticated.  As a result, a major part of this paper is to approach some of their results using the lattice model, which is inspired by  \cite{MoViWa}, \cite{Pr}, \cite{Ta}.  One can see the body part of Section \ref{even}. More precisely,    following their papers, let $K$ be a certain dyadic local field.  Let $\Psi$ be a certain non-trivial character of $K$. Let $\mathfrak{O}$ denote  the ring of integers of $K$.  Let $(\mathscr{W}, \langle, \rangle_{\mathscr{W}})$ be a vector space over $K$.  Let $\overline{\Sp(\mathscr{W})}$ denote the Metaplectic $8$-covering over $\Sp(\mathscr{W})$, attached to the Perrin-Rao's $2$-cocycle $c_{PR, \mathcal{X}^{\ast}}(-,-)$   in $\Ha^2(\Sp(\mathscr{W}), \mu_8)$ with respect to $\mathscr{W}=\mathcal{X}\oplus \mathcal{X}^{\ast}$ and $\Psi$. Let $(\Pi_{\Psi}, \mathcal{V}_{\Psi})$ be the corresponding  Weil representation of $\overline{\Sp(\mathscr{W})}$. To achieve our goal, we  consider the lattice model to realize this representation. We formulate  the precise actions of  some kind of elements in the  Metaplectic $8$-covering group instead of the  Metaplectic $\C^{\times}$-covering group,     which is  a little different  from  MVW, Takeda's works(cf.\cite{MoViWa}, \cite{Ta}).  In \cite[Thm.2]{Pr}, Prasad considered the restriction of the Weil representation to  some  open compact subgroup  in the non-dyadic local field case and obtained the irreducible components. Inspired by this result, we also consider the restriction of $\Pi_{\Psi}$ to some open compact subgroup  of $\Sp(\mathscr{W})$, and search  some subspace to realize the Weil representation as obtained by Genestier-Lysenko and Gurevich-Hadani. Following  Barthel's work \cite{Ba} in constructing the $2$-cocycle from $\Sp(\mathscr{W})$ to  $\GSp(\mathscr{W})$, we can easily  extend the Weil representation from  an affine symplectic   group to its similitude group. Some recent papers on various aspects of Weil representations can also be found here:
\cite{CMS},\cite{EhSk},\cite{Ga06},\cite{GuFrSo},\cite{GuVe},\cite{HiSc},\cite{KaTi1},\cite{KaTi2},\cite{Ro}, \cite{Sc},\cite{Ts},\cite{Zh},etc.
\section{Odd case}\label{oddcase}
Let $F=\F_q$ be  a finite field with $q=p^d$ and  $p$ is an odd prime number. Keep the notations of Introduction.  Let $\pi_{\psi}$ be the Weil representation of  $\Sp(W)\ltimes \Ha(W)$  associated to $\psi$. The representation $\pi_{\psi}$ can be realized on $\C[X]$ by the following formulas(cf. \cite[pp.65-66]{Ge}, \cite[pp.387-391]{Pe}, \cite[pp.351-360]{Ra}):
\begin{equation}\label{representationsp11}
\pi_{\psi}[1, (x,0)+(x^{\ast},0)+(0,k)]f(y)=\psi(k+\langle x+y,x^{\ast}\rangle) f(x+y),
\end{equation}
\begin{equation}\label{representationsp2}
\pi_{\psi}[ \begin{pmatrix}
  1&b\\
  0 & 1
\end{pmatrix}]f(y)=\psi(\tfrac{1}{2}\langle y,yb\rangle) f(y),
\end{equation}
\begin{equation}\label{representationsp3}
\pi_{\psi}[ \begin{pmatrix}
  a& 0\\
  0 &a^{\ast -1 }
\end{pmatrix}]f(y)=x_q^+(\det(a)) f(ya),
\end{equation}
\begin{equation}\label{representationsp4}
\pi_{\psi}[ \omega]f(y)=\gamma(\psi)^{-m}\sum_{x\in X}f(x)\psi(\langle x,y\omega\rangle),
\end{equation}
where symmetric  $b\in \Hom(X, X^{\ast})$, $a\in \Aut(X)$, and $a^{\ast} \in \Aut(X^{\ast})$ is the adjoint of $a$ with respect to the bilinear form $X \times X^{\ast} \longrightarrow F$ given by $(x, x^{\ast}) \longmapsto \langle x, x^{\ast} \rangle$, and $\{ e_1, \ldots, e_m\}$ is  an $F$-basis of $X$, and $\{ e^{\ast}_1, \ldots , e_m^{\ast}\}$  its symplectic  dual basis with respect to $\langle, \rangle$, $(e_i)\omega=-e_i^{\ast}$, $(e_i^{\ast})\omega=e_i$,  $\gamma(\psi)= \sum_{x\in F}\psi(\tfrac{x^2}{2}), x_q^{+}= \textrm{ Legendre symbol }
( \tfrac{ }{\mathbb{F}_q})$.

 Let $\lambda: \GSp(W) \longrightarrow F^{\times}$, be  the similitude factor map. Let $F^{\times}\Sp(W)=\{ h\in \GSp(W) \mid \lambda_h\in  F^{\times 2}\}=\{ tg\in \GSp(W)\mid t\in F^{\times},  g\in \Sp(W)\}$.
\subsection{ Case $q\equiv 3(\mod 4)$}  Assume $q-1=2l$, with $(2,l)=1$.  In this case,  $\chi^+_q(-1)=-1$, and $F^{\times} =F^{\times2} \times \{\pm 1\}$, $|F^{\times2}|=l$.   Let $\varpiup$ be a fixed generator of $F^{\times 2}$. Then $F^{\times 2}=\langle \varpiup^2 \rangle$. Let us define an abelian group homomorphism:
$$\tau: F^{\times 2} \longrightarrow F^{\times}; \varpiup^2 \longmapsto \varpiup^{-1}.$$
Let us consider a twisted right action of $F^{\times}\Sp(W)$ on $\Ha(W)$ in the following way:
\begin{equation}\label{alphaac}
\alpha:\Ha(W) \times F^{\times}\Sp(W) \longrightarrow \Ha(W); ((v,t), g) \longmapsto (\tau( \lambda_g)v g, t).
\end{equation}
The restriction of $\alpha$ on $\Sp(W)$ is the usual action. By observation,  the action factors through the following group homomorphism:
$$ \varsigma: F^{\times}\Sp(W) \longrightarrow \Sp(W); g \longmapsto \tau( \lambda_g)g.$$
So the restriction of $\alpha$ on $F^{\times}\Sp(W)$ is well-defined. Let us extend $\pi_{\psi}$  to $F^{\times}\Sp(W)$ by setting
 \begin{equation}\label{representationsp5}
 \pi_{\psi} [ \begin{pmatrix}
  a& 0\\
  0 &a
\end{pmatrix}] f(y)=x_q^+(a^m) f(\tau(a^2)ya), \qquad a\in F^{\times}.
\end{equation}
Then $\pi_{\psi}|_{F^{\times2}}$ is trivial. Note that  the extension is not unique, which at least can be twisted by a character of $F^{\times}/\{\pm 1\}$.
\begin{remark}
$\pi_{\psi}$ is not an irreducible representation of $F^{\times} \Sp(W)$.
\end{remark}
\begin{proof}
Assume the contrary. By Clifford theory, $\Res_{\Sp(W)}^{F^{\times }\Sp(W)}\pi_{\psi}$ only contains irreducible representations of the same dimension; this is not the fact.
\end{proof}

\begin{lemma}
$\pi_{\psi}$ is an irreducible representation of $ F^{\times}\Sp(W)\ltimes_{\alpha}\Ha(W)  $.
\end{lemma}
\begin{proof}
Let us check that the actions in  $(\ref{representationsp11})$ and $(\ref{representationsp5})$ are comparable. Take $(v,t)\in \Ha(W)$, $v=x+x^{\ast}$. Then $(v,t)=(x,0)+(x^{\ast}, 0)+(0, t-\tfrac{\langle x, x^{\ast}\rangle}{2})$. Take $g=\begin{pmatrix}
  a& 0\\
  0 &a
\end{pmatrix}$, for $a\in F^{\times}$.  As elements of $F^{\times}\Sp(W)\ltimes_{\alpha}\Ha(W)$, we have $$[g,0] [1, (v,t)] =[g,(v,t)]=[1, (\tau(\lambda_g)^{-1}vg^{-1},t)][ g, 0]=[1,(\tau(a^{-2})va^{-1},t)][ g, 0].$$
(i) $$\pi_{\psi}( [ g, 0])[\pi_{\psi}([1, (v,t)])f](y)=x_q^+(a^m)[\pi_{\psi}([1, (v,t)])f](\tau(a^2)ya)$$
$$=\psi(t+\tfrac{\langle x, x^{\ast}\rangle}{2}+\langle \tau(a^2)ya,x^{\ast}\rangle)x_q^+(a^m) f(x+\tau(a^2)ya);$$
(ii) $$\pi_{\psi}([1, (\tau(a^{-2})v a^{-1},t)])[\pi_{\psi}( [0, g]) f](y)$$
$$=\psi(t+\tfrac{1}{2}\langle \tau(a^{-2})xa^{-1},\tau(a^{-2})x^{\ast}a^{-1} \rangle+
\langle y, \tau(a^{-2})x^{\ast}a^{-1}\rangle (\pi_{\psi}( [0, g]) f)(y+\tau(a^{-2})x a^{-1})$$
$$=\psi(t+\tfrac{1}{2}\langle x,x^{\ast}\rangle+
\langle y, \tau(a^{-2})x^{\ast}a^{-1}\rangle ) x_q^+(a^m) f(\tau(a^{2})ya+x).$$
\end{proof}
\subsubsection{$\GSp(W)$ case}
Note that $ \GSp(W)/ F^{\times}\Sp(W) \simeq F^{\times}/F^{\times 2}\simeq \Z_2$.  Let us extend $\pi_{\psi}$  from  $F^{\times}\Sp(W)$ to $\GSp(W)$ by setting $\rho_{\psi}=\Ind_{F^{\times}\Sp(W)  }^{ \GSp(W)} \pi_{\psi}$.
\begin{lemma}
$\Res_{\Sp(W)}^{ \GSp(W)} \rho_{\psi} \simeq \pi_{\psi} \oplus \pi_{\psi^a}$, for any $a\in F^{\times} \setminus F^{\times 2}$.
\end{lemma}
\begin{proof}
Let $h_a= \begin{pmatrix}
1  & 0 \\
 0 & a
\end{pmatrix}$. Then $\GSp(W)=F^{\times}\Sp(W) \cup h_aF^{\times}\Sp(W)$. By Clifford theory, $\Res_{F^{\times}\Sp(W)}^{\GSp(W)} \rho_{\psi} \simeq \pi_{\psi} \oplus \pi_{\psi}^{h_a}$. It is known that $\pi_{\psi}^{h_a}|_{\Sp(W)}$ is isomorphic to  $\pi_{\psi^a}|_{\Sp(W)}$.
\end{proof}
Note that $\pi_{\psi}$ and $ \pi_{\psi^a}$ are  two different Weil representations of $\Sp(W)$.
\begin{corollary}
$\Res_{\Sp(W)}^{ \GSp(W)} \rho_{\psi}$ is independent of $\psi$, and $\rho_{\psi}$ contains two irreducible representations of $\GSp(W)$.
\end{corollary}
Let us consider the extension by adding the Heisenberg group. As $F^{\times} \simeq \{ \pm 1\} \times F^{\times 2}$, let us extend the action of $\tau$ from $F^{\times 2}$ to $F^{\times}$ by trivially acting on $\{\pm 1\}$. Let us extend  the above  twisted action from  $F^{\times}\Sp(W)$ to $\GSp(W)$  in the following way:
\begin{equation}\label{representationsp6}
\alpha:  \Ha(W)  \times \GSp(W) \longrightarrow \Ha(W); ((v,t), g) \longmapsto (\tau( \lambda_g)vg, \chi_q^+( \lambda_g)t).
\end{equation}
Let us check that it is well-defined. Note that $\GSp(W)=[F^{\times}\Sp(W)]h_{-1}=h_{-1}[F^{\times}\Sp(W)]$, for $h_{-1}=\begin{pmatrix}
  1& 0\\
  0 &-1
\end{pmatrix}$.  For $(v,t), (v',t')\in \Ha(W)$, $g\in  F^{\times}\Sp(W)$,

(i)
$$(v,t)\alpha(h_{-1})\alpha(h_{-1})=(v,t);$$

(ii)
$$(v,t)\alpha(g)\alpha(h_{-1})=[\tau( \lambda_g)vg, t]\alpha(h_{-1})=[\tau( \lambda_{h_{-1}})\tau( \lambda_g)vgh_{-1}, -t]$$
$$=[\tau( \lambda_{gh_{-1}})vgh_{-1}, -t]=[v,t]\alpha(gh_{-1});$$
(iii)
$$(v,t)\alpha(h_{-1})\alpha(g)\alpha(h_{-1})=[vh_{-1}, -t]\alpha(g)\alpha(h_{-1})=[\tau( \lambda_g)vh_{-1}g, -t]\alpha(h_{-1})$$
$$=[\tau( \lambda_g)vh_{-1}gh_{-1}, t]=[v,t]\alpha(h_{-1}gh_{-1});$$
(iv)
$$[(v,t)+(v',t')]\alpha(h_{-1})=[(v+v', t+t'+\tfrac{\langle v, v'\rangle}{2})]\alpha(h_{-1})=(-vh_{-1}-v'h_{-1},- t-t'-\tfrac{\langle v, v'\rangle}{2})];$$
$$(v,t)\alpha(h_{-1})+(v',t')\alpha(h_{-1})=(-v h_{-1}, -t) +(-v'h_{-1},-t')=(-v h_{-1}-v'h_{-1},-t-t'+\tfrac{\langle -v h_{-1},-v'h_{-1}\rangle}{2})$$
$$=(-vh_{-1}-v'h_{-1},- t-t'-\tfrac{\langle v, v'\rangle}{2})=[(v,t)+(v',t')]\alpha(h_{-1}).$$
\begin{lemma}
$\rho_{\psi}$ can extend to be  an irreducible representation of $ \GSp(W) \ltimes_{\alpha}\Ha(W)$.
\end{lemma}
\begin{proof}
Let us consider $\rho_{\psi}'=\Ind_{F^{\times}\Sp(W)\ltimes_{\alpha} \Ha(W)  }^{\GSp(W) \ltimes_{\alpha}\Ha(W)} \pi_{\psi}$. By Clifford theory, its restriction to $\GSp(W)$  is isomorphic with $\rho_{\psi}$. Moreover, $\rho_{\psi}'|_{F^{\times}\Sp(W) \ltimes_{\alpha} \Ha(W)  }$ contains two irreducible representations, and $\rho_{\psi}'$ extends $\pi_{\psi}\oplus \pi_{\psi^{-1}}$. Therefore, $\rho_{\psi}'|_{  F^{\times}\Sp(W) \ltimes_{\alpha}\Ha(W)}$ contains two different irreducible components. Then  $\rho_{\psi}'$ is irreducible.
\end{proof}
Let us define $\PGSp^{\pm}(W)= \GSp(W) /{F^{\times 2}}$. Then there exists the following exact sequence:
$$ 1 \longrightarrow \Sp(W) \longrightarrow  \PGSp^{\pm}(W) \longrightarrow F^{\times}/ F^{\times 2} \longrightarrow 1.$$
Moreover, the above action $\alpha$ factors through $\GSp(W) \longrightarrow \PGSp^{\pm}(W)  $, and  $\pi_{\psi}$ is a representation of $\Sp(W)\ltimes_{\alpha} \Ha(W)$. Hence, $\rho_{\psi}' $(resp. $\rho_{\psi}$) is indeed an  irreducible representation of $ \PGSp^{\pm}(W)\ltimes_{\alpha}\Ha(W)$(resp. $\PGSp^{\pm}(W)$). Moreover, the restriction of $\rho_{\psi}$ on $\Sp(W)$ consists of the two different Weil representations. Hence, it is independent of the choice of  $\psi$.

\subsection{ Case $q\equiv 1(\mod 4)$ I} Assume $q-1=2^n l$,  $2\nmid l$, and $n>1$. So $F^{\times}=F_l \times  F_{2^n}  $,  for some cyclic subgroups $F_l$, $F_{2^{n}}$  of order $l$, and $2^{n}$ respectively. Then $F_l\subseteq F^{\times 2}$. Let $\varpiup$ be a fixed generator of $F_l$. Then $F_l=\langle \varpiup^2 \rangle$. Let us define an abelian group homomorphism:
$$\tau: F_l \longrightarrow F^{\times}; \varpiup^2 \longmapsto \varpiup^{-1}.$$
As $F^{\times} =F_l \times  F_{2^n}$, $\tau$ can extend to $F^{\times}$ by trivially acting on $F_{2^n}$.
Let $\chi^+$ be the projection map from $F^{\times}$ to $F_{2^{n}}$. Clearly, $\chi^+$ is a group homomorphism.  Let us define  a twisted action of  $\GSp(W)$ on $\Ha(W)$ in the following way:
\begin{equation}\label{representationsp7}
\alpha:  \Ha(W) \times \GSp(W)\longrightarrow \Ha(W); [ (v,t), g] \longmapsto (\tau( \lambda_g)vg, \chi^+( \lambda_g)t).
\end{equation}
On $F_l\Sp(W)$, this  action factors through the following group homomorphism:
$$ \varsigma: F_l\Sp(W) \longrightarrow \Sp(W); g \longmapsto \tau( \lambda_g)g.$$
So the Weil representation  $\pi_{\psi}$  can extend to $F_l\Sp(W)\ltimes_{\alpha}\Ha(W)$. Let us define
$$\rho_{\psi}'=\Ind_{F_l\Sp(W)\ltimes_{\alpha} \Ha(W)  }^{\GSp(W) \ltimes_{\alpha}\Ha(W)} \pi_{\psi}.$$
\begin{lemma}
$\rho_{\psi}'$ is an irreducible representation.
\end{lemma}
\begin{proof}
 Assume $ \GSp(W)=\oplus_{t\in F_{2^n}}  h_t[F_l\Sp(W)]$, for $h_t=\begin{pmatrix}
  1& 0\\
  0 &t
\end{pmatrix}$. The restriction of $\rho_{\psi}'$ to $F^{\times}\Sp(W)\ltimes_{\alpha}\Ha(W) $ contains $2^{n}$-irreducible components of the forms $\pi_{\psi}^{h_t}$. Hence  $\rho_{\psi}'|_{\Sp(W)\ltimes \Ha(W) }\simeq \oplus_{t\in F_{2^n}}\pi_{\psi}^{h_t}$. By the above action,  $\pi_{\psi}^{h_t}|_{ \Sp(W)\ltimes \Ha(W)} \simeq \pi_{\psi^{t}}$. Notice that $\pi_{\psi^{t}}$ are different, for different $t$. So these $\pi_{\psi}^{h_t}$ are different irreducible representations. By Clifford theory,  $\rho_{\psi}'$ is irreducible.
\end{proof}
Note that $\rho_{\psi}'|_{\Sp(W)}$ contains more components than  $[\pi_{\psi}\oplus\pi_{\psi^{a}}]|_{\Sp(W)} $, for some $a\in F^{\times}\setminus F^{\times 2}$.

\subsection{ Case $q\equiv 1(\mod 4)$ II}
 Let us consider another possible way to extend the Weil representation by some twisted actions.
  \subsubsection{$\widetilde{F^{\times}}$} Let us firstly consider $F^{\times}$ and its subgroup $F^{\times 2}$.  Let us fix    a square root map:
$$\sqrt{\quad}: F^{\times 2} \longrightarrow F^{\times},$$ such that $(\sqrt{a})^2=a$.  For $a, b \in F^{\times 2}$,
$\sqrt{a}\sqrt{b} =c_{\sqrt{\ }}(a,b)\sqrt{ab}$, for some $c_{\sqrt{\ }}(a,b) \in \{ \pm 1 \}$. Assume $\sqrt{1}=1$. Then:
$$\sqrt{a}\sqrt{b}\sqrt{c}=c_{\sqrt{\ }}(a,b)\sqrt{ab}\sqrt{c}=c_{\sqrt{\ }}(a,b)c_{\sqrt{\ }}(ab, c)\sqrt{abc}=c_{\sqrt{\ }}(b,c)\sqrt{a}\sqrt{bc}=c_{\sqrt{\ }}(b,c)c_{\sqrt{\ }}(a,bc)\sqrt{abc},$$
$$\sqrt{1}\sqrt{b}=\sqrt{b}=c_{\sqrt{\ }}(1,b)\sqrt{b}=c_{\sqrt{\ }}(b,1)\sqrt{b}\sqrt{1}.$$
Hence $c_{\sqrt{\ }}(-,-)$ defines a $2$-cocycle from $F^{\times 2}\times F^{\times 2}$ to $\{ \pm 1 \}$. Let $\widetilde{F^{\times 2}}$ be the corresponding  central extension of $F^{\times 2} $ by $\{ \pm 1 \}$. Since $c_{\sqrt{\ }}$ is symmetric, $\widetilde{F^{\times 2}}$ is an abelian group. Then there exists a group homomorphism:
$$\sqrt{\quad}: \widetilde{F^{\times 2}} \longrightarrow F^{\times}, [a, \epsilon] \longmapsto \sqrt{a} \epsilon.$$
If $\sqrt{a} \epsilon=1$, then $\sqrt{a}=\pm 1$, and then $a=1$, $\sqrt{a}=1$, $\epsilon=1$. So it is also an injective map. Since both sides are finite groups, it is an isomorphism by comparing the orders.

 It is known that there exists a short exact sequence of groups:
\begin{equation}\label{fff}
1\longrightarrow F^{\times 2} \longrightarrow F^{\times} \longrightarrow F^{\times}/F^{\times 2} \longrightarrow 1.
\end{equation}
For an element $t\in F^{\times}$, let $\dot{t}$  denote its reduction in $F^{\times}/F^{\times 2}$.  It is known that $ F^{\times}/F^{\times 2}$ has order $2$. Assume $ F^{\times}/F^{\times 2}=\{ \dot{1}, \dot{\xi}\}$. Note that $F^{\times} \simeq F_l\times F_{2^n}$, and $F_l\subseteq F^{\times 2}$. So we choose $\xi \in F_{2^n}$.
Let $\kappa$ be the canonical section map from $F^{\times}/F^{\times 2}$ to $F^{\times}$ such that $\kappa( \dot{1})=1$, $\kappa(\dot{\xi})=\xi$.

By (\ref{fff}),  $F^{\times}$ can be viewed as a central extension of $F^{\times}/F^{\times 2}$ by $F^{\times 2}$. Let $c'(-,-)$ denote the $2$-cocycle associated to $\kappa$. More precisely, for two elements $t_1, t_2\in F^{\times}$, let us write $t_1=a_1^2 \kappa(\dot{t}_1)$, $t_2=a_2^2 \kappa(\dot{t}_2)$. Then:
 $$t_1t_2=(a_1a_2)^2 \kappa(\dot{t}_1) \kappa(\dot{t}_2)=(a_1a_2)^2 c'(\dot{t}_1, \dot{t}_2)\kappa(\dot{t}_1\dot{t}_2).$$
 Hence $c'(\dot{t}_1, \dot{t}_2)=c'(\dot{t}_2, \dot{t}_1)$, and $c'(\dot{1}, \dot{t}_2)=1$,  $c'(\dot{\xi}, \dot{\xi})=\xi^2$.
   Recall that there exists a canonical map: $$\iota:F^{\times 2} \longrightarrow \widetilde{F^{\times 2}}; a \longmapsto [a,1].$$
Through $\iota$, we view $c'(-,-)$ as a map from $F^{\times}/F^{\times 2} \times F^{\times}/F^{\times 2}$ to $\widetilde{F^{\times 2}}$. It can be checked that this map also  defines a $2$-cocycle. We denote it by $c''(-,-)$ from now on.  Associated to this $2$-cocycle, there exists an exact sequence:
\begin{equation}\label{fffww}
1\longrightarrow \widetilde{F^{\times 2}} \longrightarrow \widetilde{F^{\times}} \longrightarrow F^{\times}/F^{\times 2} \longrightarrow 1.
\end{equation}
\begin{lemma}
There exists a group homomorphism:
$$ \widetilde{F^{\times }} \longrightarrow F^{\times}; ([g,\epsilon], \dot{t}) \longmapsto [g, \dot{t}]$$ such that
 the following  commutative diagram holds:
\[
\begin{CD}
 @. 1@. 1 @.1 @. \\
@.@VVV @VVV  @VVV \\
1 @>>> \{\pm 1\}@>>>   \{\pm 1\} @>>>1 @>>> 1\\
@.@VVV @VVV  @VVV \\
1 @>>> \widetilde{F^{\times 2}}@>>>  \widetilde{F^{\times}} @>>>F^{\times}/F^{\times 2} @>>> 1\\
@.@VVV @VVV @VV{=}V \\
1 @>>> F^{\times 2}@>>> F^{\times} @>>>F^{\times}/F^{\times 2} @>>> 1\\
@.@VVV @VVV  @VVV \\
 @. 1@. 1 @.1 @. \\
\end{CD}
\]
\end{lemma}
\begin{proof}
It is clear that the map satisfies the commutative property. So it suffices to show that  the map is also a group homomorphism. For $([g_1, \epsilon_1], \dot{t}_1)$, $([g_1, \epsilon_1], \dot{t}_1) \in  \widetilde{F^{\times}} $,
$$([g_1, \epsilon_1], \dot{t}_1)([g_2, \epsilon_2], \dot{t}_2)=([g_1, \epsilon_1][g_2, \epsilon_2]c''( \dot{t}_1,  \dot{t}_2), \dot{t}_1\dot{t}_2)$$
$$=([g_1, \epsilon_1][g_2, \epsilon_2][c'( \dot{t}_1,  \dot{t}_2), 1], \dot{t}_1\dot{t}_2)=([g_1g_2c'(\dot{t}_1,\dot{t}_2),c_{\sqrt{\ }}(g_1, g_2)c_{\sqrt{\ }}(g_1g_2, c'(\dot{t}_1,\dot{t}_2))], \dot{t}_1\dot{t}_2).$$
On the other hand, in $F^{\times}$:
$$[g_1, \dot{t}_1][g_2, \dot{t}_2]=[g_1g_2c'(\dot{t}_1,\dot{t}_2), \dot{t}_1\dot{t}_2].$$
\end{proof}
Moreover, associated to  the middle column of the above commutative diagram, it is a central extension of $F^{\times}$ by $\{\pm 1\}$. Let us denote the corresponding  $2$-cocycle by $c'''(-,-)$. Let us write down the explicit expression as follows:

For  $t_1=a_1^2 \kappa(\dot{t}_1)$, $t_2=a_2^2 \kappa(\dot{t}_2)\in F^{\times}$, they correspond to two elements
$([a_1^2,1], \dot{t}_1)$, $([a_2^2,1], \dot{t}_2)$ of $\widetilde{F^{\times}}$, with $ [a_1^2,1], [a_2^2,1]\in \widetilde{F^{\times 2}}$, and $\dot{t}_1, \dot{t}_2\in F^{\times}/F^{\times 2}$. Then:
$$t_1t_2=(a_1a_2)^2 \kappa(\dot{t}_1) \kappa(\dot{t}_2)=(a_1a_2)^2 c'(\dot{t}_1, \dot{t}_2)\kappa(\dot{t}_1\dot{t}_2).$$
Hence $t_1t_2$ corresponds to the element
$([(a_1a_2)^2 c'(\dot{t}_1, \dot{t}_2),1], \dot{t}_1\dot{t}_2)$ of $\widetilde{F^{\times}}$, with  $ [(a_1a_2)^2 c'(\dot{t}_1, \dot{t}_2),1]\in \widetilde{F^{\times 2}}$, and $\dot{t}_1\dot{t}_2\in F^{\times}/F^{\times 2}$. In $\widetilde{F^{\times}}$:
\begin{equation}
\begin{split}
([a_1^2,1], \dot{t}_1)([a_2^2,1], \dot{t}_2)&=([a_1^2,1][a_2^2,1]c''(\dot{t}_1,\dot{t}_2), \dot{t}_1\dot{t}_2)\\
 &=([a_1^2,1][a_2^2,1][c'(\dot{t}_1,\dot{t}_2),1], \dot{t}_1\dot{t}_2)\\
 &=\big(a_1^2a_2^2c'(\dot{t}_1,\dot{t}_2),c_{\sqrt{\ }}(a_1^2, a_2^2)c_{\sqrt{\ }}(a_1^2a_2^2, c'(\dot{t}_1,\dot{t}_2)), \dot{t}_1\dot{t}_2\big).
 \end{split}
 \end{equation}
On the other hand, $$[t_1, 1][t_2,1]=[t_1t_2, c'''(t_1, t_2)].$$
Therefore,
$$c'''(t_1,t_2)=c_{\sqrt{\ }}(a_1^2, a_2^2)c_{\sqrt{\ }}(a_1^2a_2^2, c'(\dot{t}_1,\dot{t}_2)),$$
for $t_1=a_1^2 \kappa(\dot{t}_1)$, $t_2=a_2^2 \kappa(\dot{t}_2) \in F^{\times}$.
\subsubsection{$\widetilde{\GSp}(W)$} Let us extend these $2$-cocycles to $\GSp(W)$.
It is known that there exist two   exact sequences of groups:
$$1\longrightarrow \Sp(W) \longrightarrow \GSp(W) \longrightarrow F^{\times} \longrightarrow 1, $$
$$1\longrightarrow F^{\times}\Sp(W) \longrightarrow \GSp(W) \longrightarrow F^{\times}/F^{\times 2} \longrightarrow 1.$$
Through the projection $\GSp(W) \longrightarrow F^{\times}$, we lift $c'''(-,-)$ from $F^{\times}$ to $\GSp(W)$. Associated to this cocycle, there exists an extension  of $\GSp(W)$ by $\{\pm 1\}$:
$$1 \longrightarrow \{\pm 1\} \longrightarrow \widetilde{\GSp}(W) \longrightarrow \GSp(W) \longrightarrow 1.$$
Composed with $\GSp(W) \longrightarrow F^{\times}/F^{\times 2}$, there  is  a group homomorphism: $ \widetilde{\GSp}(W) \longrightarrow F^{\times}/F^{\times 2}$. Let $\widetilde{F^{\times }\Sp(W)}$ denote its kernel. Then there exists the following commutative diagram:
\[
\begin{CD}
 @. 1@. 1 @.1 @. \\
@.@VVV @VVV  @VVV \\
1 @>>> \{\pm 1\}@>>>   \{\pm 1\} @>>>1 @>>> 1\\
@.@VVV @VVV  @VVV \\
1 @>>> \widetilde{F^{\times }\Sp(W)}@>>>  \widetilde{\GSp}(W) @>>>F^{\times}/F^{\times 2} @>>> 1\\
@.@VVV @VVV @VV{=}V \\
1 @>>> F^{\times}\Sp(W)@>>> \GSp(W)  @>>>F^{\times}/F^{\times 2} @>>> 1\\
@.@VVV @VVV  @VVV \\
 @. 1@. 1 @.1 @. \\
\end{CD}
\]
\begin{lemma}
There exists a group homomorphism:
$\widetilde{\lambda}: \widetilde{\GSp}(W)  \longrightarrow \widetilde{F^{\times}}; [g, \epsilon] \longmapsto [\lambda_g, \epsilon]$.
\end{lemma}
\begin{proof}
For $[g_1, \epsilon_1] , [g_2, \epsilon_2] \in \widetilde{\GSp}(W)$, $$[g_1, \epsilon_1] [g_2, \epsilon_2]=[g_1g_2, c'''(\lambda_{g_1},\lambda_{g_2}) \epsilon_1 \epsilon_2].$$
So $$\widetilde{\lambda}([g_1, \epsilon_1] )\widetilde{\lambda}([g_2, \epsilon_2]  )=[\lambda_{g_1}, \epsilon_1][\lambda_{g_2}, \epsilon_2]=[\lambda_{g_1}\lambda_{g_2}, c'''(\lambda_{g_1},\lambda_{g_2}) \epsilon_1 \epsilon_2],$$
$$\widetilde{\lambda}([g_1, \epsilon_1] [g_2, \epsilon_2])=\widetilde{\lambda}([g_1g_2, c'''(\lambda_{g_1},\lambda_{g_2}) \epsilon_1 \epsilon_2])=[\lambda_{g_1g_2}, c'''(\lambda_{g_1},\lambda_{g_2}) \epsilon_1 \epsilon_2].$$
\end{proof}
Note that the projection  $ \widetilde{\GSp}(W) \longrightarrow F^{\times}/F^{\times 2}$,  factors through $\widetilde{\lambda}$.
\begin{lemma}\label{ste}
\begin{itemize}
\item[(1)] $\Sp(W) \simeq \ker(\widetilde{\lambda}); g \longmapsto [g,1]$.
\item[(2)] There exists a group monomorphsim: $\nu:\widetilde{F^{\times 2}} \longrightarrow \widetilde{\GSp}(W); [a, \epsilon] \longrightarrow [\sqrt{a}\epsilon, \epsilon]$.   Moreover the image lies in the center of  $\widetilde{\GSp}(W)$, and $ \widetilde{F^{\times2}} \cap \Sp(W)=\{1\}$.
\item[(3)] The above group $\widetilde{F^{\times }\Sp(W)} =\widetilde{F^{\times 2}} \Sp(W)$.
\end{itemize}
\end{lemma}
\begin{proof}
1) $\ker(\widetilde{\lambda})=\{ [g, 1]\mid g\in \Sp(W)\}$. Moreover, $c'''(\lambda_{g_1}, \lambda_{g_2})=1$, for any $g_1, g_2\in \Sp(W)$. So the result holds.\\
2) For $[a_i, \epsilon_i]\in \widetilde{F^{\times 2}}$,
$$[a_1, \epsilon_1][a_2, \epsilon_2]=[a_1a_2, c_{\sqrt{\ }}(a_1,a_2)\epsilon_1\epsilon_2];$$
$$ [\sqrt{a}_1\epsilon_1, \epsilon_1] [\sqrt{a}_2\epsilon_2, \epsilon_2]=[\sqrt{a}_1\epsilon_1\sqrt{a}_2\epsilon_2, c'''(\lambda_{\sqrt{a}_1\epsilon_1},\lambda_{\sqrt{a}_2\epsilon_2}) \epsilon_1 \epsilon_2]$$
$$=[\sqrt{a}_1\epsilon_1\sqrt{a}_2\epsilon_2, c_{\sqrt{\ }}(a_1,a_2) \epsilon_1 \epsilon_2];$$
$$[\sqrt{a_1a_2}c_{\sqrt{\ }}(a_1,a_2)\epsilon_1\epsilon_2, c_{\sqrt{\ }}(a_1,a_2)\epsilon_1\epsilon_2]=[\sqrt{a}_1\sqrt{a}_2\epsilon_1\epsilon_2,  c_{\sqrt{\ }}(a_1,a_2)\epsilon_1\epsilon_2].$$
If $[1,1]=\nu([a, \epsilon])=[\sqrt{a}\epsilon, \epsilon]$, then $\sqrt{a}=1=\epsilon$. So $\nu$ is an injective map. For
$[g_1,\epsilon_1]\in \widetilde{\GSp}(W)$, $[\sqrt{a}\epsilon, \epsilon]\in \nu(\widetilde{F^{\times2}})$,
$$[g_1,\epsilon_1][\sqrt{a}\epsilon, \epsilon]=[g_1\sqrt{a}\epsilon, c'''(\lambda_{g_1},\lambda_{\sqrt{a}\epsilon})\epsilon_1\epsilon]$$
$$=[g_1\sqrt{a}\epsilon, c'''(\lambda_{g_1},a)\epsilon_1\epsilon];$$
$$[\sqrt{a}\epsilon, \epsilon][g_1,\epsilon_1]=[\sqrt{a}\epsilon g_1, c'''(a,\lambda_{g_1})\epsilon_1\epsilon].$$
Since $g_1\sqrt{a}\epsilon=\sqrt{a}\epsilon g_1$, $ c'''(\lambda_{g_1},a)=c'''(a,\lambda_{g_1})$, $[\sqrt{a}\epsilon, \epsilon]$ lies in the center of $\widetilde{\GSp}(W)$.
If $[\sqrt{a}\epsilon, \epsilon] \in \Sp(W)$,  then $\epsilon=1$, and $\sqrt{a}\epsilon\in \Sp(W)$. Hence $\sqrt{a}=\pm 1$, $a=1$, $\sqrt{a}=1$. \\
3) For $[g, \epsilon]\in \widetilde{\GSp}(W)$, the image of it in $F^{\times}/F^{\times 2}$ equals to $\dot{\lambda}_{g}$. So $\dot{\lambda}_{g}=\dot{1}$ iff $g\in F^{\times}\Sp(W)$. For any $th\in F^{\times} \Sp(W)$, $\epsilon\in \{\pm 1\}$, $[th, \epsilon]=[\pm 1 \cdot\sqrt{t^2} h,\epsilon]=[\sqrt{t^2} \epsilon, \epsilon][\pm \epsilon h, 1]$. Hence $\widetilde{F^{\times }\Sp(W)} \subseteq \widetilde{F^{\times 2}} \Sp(W)$. It is clear that the other inclusion holds.
\end{proof}
Let us define $$\widetilde{F^{\times}}_+=\nu(\widetilde{F^{\times 2}}) \subseteq  \widetilde{\GSp}(W) \textrm{ and }  \PGSp^{\pm}(W)=\widetilde{\GSp}(W)/\widetilde{F^{\times}}_+.$$
Then there exist the following exact sequences of groups:
$$1 \longrightarrow \widetilde{F^{\times 2}} \Sp(W) \longrightarrow \widetilde{\GSp}(W) \stackrel{\dot{\widetilde{\lambda}}}{\longrightarrow}\widetilde{F^{\times}} / \widetilde{F^{\times 2}} \simeq F^{\times}/F^{\times 2} \longrightarrow 1; $$
$$1 \longrightarrow  \Sp(W) \longrightarrow \PGSp^{\pm}(W) \stackrel{\dot{\widetilde{\lambda}}}{\longrightarrow}\widetilde{F^{\times}} / \widetilde{F^{\times 2}} \simeq F^{\times}/F^{\times 2} \longrightarrow 1.$$
Recall  the Weil representation $\pi_{\psi}$ of $ \Sp(W)$. Let us define:
$$\rho_{\psi}=\Ind_{\Sp(W)}^{\PGSp^{\pm}(W)} \pi_{\psi}.$$
By Clifford theory, the restriction of $\rho_{\psi}$ on $\Sp(W)$ consists of the two different Weil representations. Hence, it is independent of the choice of  $\psi$.

\subsubsection{$\wideparen{\widetilde{\GSp}}(W)$} Recall  that $c''(-,-)$ is a $2$-cocycle from $F^{\times}/F^{\times 2} \times F^{\times}/F^{\times 2} $ to $\widetilde{F^{\times2}}$.
Through the projection $\widetilde{\GSp}(W)\longrightarrow F^{\times}/F^{\times 2} $, we lift the $2$-cocycle $c''(-,-)$ to define over $\widetilde{\GSp}(W)$. For $\widetilde{g}_1=[g_1, \epsilon_1], \widetilde{g}_2=[g_2, \epsilon_2]\in \widetilde{\GSp}(W)$, we define
 $$c(\widetilde{g}_1, \widetilde{g}_2)= \sqrt{ c''(\dot{\lambda}_{g_1}, \dot{\lambda}_{g_2})}.$$

From the definition, we know that  $c(\widetilde{g}_1, \widetilde{g}_2)=1=c(\widetilde{g}_2, \widetilde{g}_1)$, for $\widetilde{g}_1\in \widetilde{F^{\times 2
}} \Sp(W)$. Associated to $c(-,-)$, there is  a central extension of $\widetilde{\GSp}(W)$ by $F^{\times}$:
$$1 \longrightarrow F^{\times} \longrightarrow \wideparen{\widetilde{\GSp}}(W)  \longrightarrow  \widetilde{\GSp}(W) \longrightarrow 1.$$
Then the above extension is split over $\widetilde{F^{\times 2}}\Sp(W)$. The group   $\wideparen{\widetilde{\GSp}}(W)$ consists of the elements $[\widetilde{g}, k]$, $\widetilde{g}\in \widetilde{\GSp}(W), k\in F^{\times}$, and  the law is  given as follows:
for $\widetilde{g}$, $\widetilde{g}'\in \widetilde{\GSp}(W)$, $k,k'\in F^{\times}$,
  $$[\widetilde{g}, k] \ast [\widetilde{g}',k']=[\widetilde{g}\widetilde{g}' , c(\widetilde{g},\widetilde{g}')kk'].$$
Note that $[F^{\times}\times \widetilde{F^{\times 2}}\Sp(W)] \unrhd \wideparen{\widetilde{\GSp}}(W)$, $\wideparen{\widetilde{\GSp}}(W)/[F^{\times}\times \widetilde{F^{\times 2}}\Sp(W)] \simeq F^{\times}/F^{\times 2}$. Moreover, $F^{\times}\times \widetilde{F^{\times 2}} $ lies in the center of $\wideparen{\widetilde{\GSp}}(W) $.

 For any element $\widetilde{g}=[g, \epsilon]\in \widetilde{\GSp}(W)$, let us write $\lambda_g=a_{g}^2\kappa(\dot{\lambda}_g)$, for some $a_{g} \in F^{\times }$.  Then under the map $\widetilde{\GSp}(W) \stackrel{\widetilde{\lambda}}{\longrightarrow} \widetilde{F^{\times}}$, the image of $\widetilde{g}$ is just $[\lambda_g, \epsilon]$. By the exact sequence(\ref{fffww}), let us write $[\lambda_g, \epsilon]$ in the other form: $([a_{g}^2, \epsilon], \dot{\lambda}_g)$ with $[a_{g}^2, \epsilon] \in \widetilde{F^{\times 2}}$, $\dot{\lambda}_g\in F^{\times}/F^{\times 2}$. Let us define a map:
 $$\tau: \widetilde{\GSp}(W) \longrightarrow F^{\times}; \widetilde{g}\longmapsto (\sqrt{[a_{g}^2, \epsilon]})^{-1}.$$
\begin{lemma}
Let  $ \widetilde{g}_1=[g_1, \epsilon_1]$,$ \widetilde{g}_2=[g_2, \epsilon_2]\in \widetilde{\GSp}(W)$.
\begin{itemize}
\item[(1)] $\tau(\widetilde{g}_1)\tau(\widetilde{g}_2)=\tau(\widetilde{g}_1\widetilde{g}_2) c(\widetilde{g}_1, \widetilde{g}_2)$.
\item[(2)]  $\lambda_{\tau(\widetilde{g}_1)g_1} \lambda_{\tau(\widetilde{g}_2)g_2} =\lambda_{\tau(\widetilde{g}_1\widetilde{g}_2)g_1g_2} c(\widetilde{g}_1, \widetilde{g}_2)^2$.
\end{itemize}
\end{lemma}
\begin{proof}
1) For $\widetilde{\lambda}_{\widetilde{g}_i}=([a_{g_i}^2, \epsilon_i], \dot{\lambda}_{g_i})\in \widetilde{F^{\times}} $, $i=1,2$, we have
$$([a_{g_1}^2, \epsilon_1], \dot{\lambda}_{g_1})([a_{g_2}^2, \epsilon_2], \dot{\lambda}_{g_2})=([a_{g_1}^2, \epsilon_1][a_{g_2}^2, \epsilon_2] c''(\dot{\lambda}_{g_1},\dot{\lambda}_{g_2}), \dot{\lambda}_{g_1}\dot{\lambda}_{g_2}).$$
So $\tau(\widetilde{g}_i)=\sqrt{[a_{g_i}^2, \epsilon_i]}^{-1}$,
$$\tau(\widetilde{g}_1\widetilde{g}_2)=\sqrt{[a_{g_1}^2, \epsilon_1][a_{g_2}^2, \epsilon_2] c''(\dot{\lambda}_{g_1},\dot{\lambda}_{g_2})}^{-1}=\sqrt{[a_{g_1}^2, \epsilon_1]}^{-1}\sqrt{[a_{g_2}^2, \epsilon_2]}^{-1}\sqrt{c''(\dot{\lambda}_{g_1},\dot{\lambda}_{g_2})}^{-1}$$
$$=\tau(\widetilde{g}_1)\tau(\widetilde{g}_2) c(\widetilde{g}_1, \widetilde{g}_2)^{-1}.$$
2) It is a consequence of (1).
\end{proof}
For our purpose, let us extend the action on $\Ha(W)$ from $\Sp(W)$  to $\wideparen{\widetilde{\GSp}}(W) $. Recall that there exists a  group homomorphism: $\widetilde{\GSp}(W) \longrightarrow \GSp(W)$. Through this map, we have a canonical action of $\widetilde{\GSp}(W)$ on $W$ as well as  $\Ha(W)$.

 Let us consider a twisted right action of $\wideparen{\widetilde{\GSp}}(W) $ on $\Ha(W)$ in the following way:
\begin{equation}\label{GSPH}
\alpha:\Ha(W) \times \wideparen{\widetilde{\GSp}}(W) \longrightarrow \Ha(W); ((v,t), [\widetilde{g},k]) \longmapsto  (\tau(\widetilde{g})v gk, \lambda_{\tau(\widetilde{g})g} \lambda_{k}t),
\end{equation}
for $\widetilde{g}=[g,\epsilon]\in \widetilde{\GSp}(W)$ with $g\in \GSp(W)$, $\epsilon \in \{\pm 1\}$, $k\in F^{\times}$.
\begin{lemma}
The above action $\alpha$ is well-defined.
\end{lemma}
\begin{proof}
 For $ [\widetilde{g}, k] \in \wideparen{\widetilde{\GSp}}(W) $, $(v, t), (v', t') \in   \Ha(W)$, it can be checked that  $[(v, t)+ (v', t')]\alpha([\widetilde{g}, k])=(v, t)\alpha([\widetilde{g},k])+  (v', t')\alpha([\widetilde{g},k])$. For $[1_{\widetilde{\GSp}(W)},1_{F^{\times}}]\in \wideparen{\widetilde{\GSp}}(W)$, it is clear that $(v, t) \alpha([1_{\widetilde{\GSp}(W)},1_{F^{\times}}])=(v,t)$. Note that $1_{\widetilde{\GSp}(W)} \times F^{\times}$ lies in the center of $\wideparen{\widetilde{\GSp}}(W) $. Moreover, $\alpha|_{1_{\widetilde{\GSp}(W)} \times F^{\times}}$ is well-defined.
For $\widetilde{g}_1=[g_1, \epsilon_1], \widetilde{g}_2=[g_2,\epsilon_2] \in \widetilde{\GSp}(W)$,
 $$(v,t) \alpha([\widetilde{g}_1,1])\alpha([\widetilde{g}_2,1])=(\tau(\widetilde{g}_1)v g_1, \lambda_{\tau(\widetilde{g}_1)g_1} t)\alpha(\widetilde{g}_2)=(\tau(\widetilde{g}_1)\tau(\widetilde{g}_2)v g_1g_2, \lambda_{\tau(\widetilde{g}_1)g_1}\lambda_{\tau(\widetilde{g}_2)g_2} t),$$
 $$(v,t) \alpha([\widetilde{g}_1,1]\ast[\widetilde{g}_2,1])=(v,t) \alpha([\widetilde{g}_1\widetilde{g}_2,c(\widetilde{g}_1,\widetilde{g}_2)])=(\tau(\widetilde{g}_1\widetilde{g}_2)v g_1g_2c(\widetilde{g}_1,\widetilde{g}_2), \lambda_{\tau(\widetilde{g}_1\widetilde{g}_2)g_1g_2}c(\widetilde{g}_1,\widetilde{g}_2)^{2}t)$$
 $$=(v,t) \alpha([\widetilde{g}_1,1])\alpha([\widetilde{g}_2,1]).$$
\end{proof}
By Lemma \ref{ste}, we can view $\Sp(W)$ as a subgroup of $\widetilde{\GSp}(W)$, and then as a subgroup of $\wideparen{\widetilde{\GSp}}(W)$. The restriction of the action $\alpha$ on $\Sp(W)$ is the usual action.

Moreover, for  $\widetilde{g}=[g, \epsilon]\in \widetilde{F^{\times2}} \Sp(W)$, $\lambda_{g}\in F^{\times 2}$, $\lambda_{\tau(\widetilde{g}) g}=1$.   Hence on $  \widetilde{F^{\times2}} \Sp(W)$, the  action $\alpha$  factors through $\Sp(W)$. Let  $\Z_2=\{ \pm 1\}$.
Recall  the Weil representation $\pi_{\psi}$ of $ \Sp(W)\ltimes \Ha(W)$. Then $\pi_{\psi}$ can be viewed as a representation of $[\Z_2\times \widetilde{F^{\times2}} \Sp(W)]\ltimes_{\alpha} \Ha(W)$.
Let us consider $\rho'_{\psi}=\Ind_{[\Z_2\times \widetilde{F^{\times2}} \Sp(W)]\ltimes_{\alpha} \Ha(W)  }^{\wideparen{\widetilde{\GSp}}(W) \ltimes_{\alpha}\Ha(W)} \pi_{\psi}$.
\begin{lemma}
\begin{itemize}
\item[(1)] $\rho'_{\psi}|_{\Sp(W) \ltimes \Ha(W)}$ contains $\pi_{\psi^s}$, for any $s\in F^{\times}$.
\item[(2)]  $\rho'_{\psi}$ is independent of the choice of $\psi$, i.e. $\rho'_{\psi}\simeq \rho_{\psi^s}$, for any $s\in F^{\times}$.
\end{itemize}
\end{lemma}
\begin{proof}
1) Let us write $s=a_s^2 \kappa(\dot{s})$. Let $g_s=\begin{pmatrix}
1 & 0\\
0 & s\end{pmatrix} \in \GSp(W)$,  $\widetilde{g}=[g_s, 1]\in \widetilde{\GSp}(W)$, $k=\tau(\widetilde{g})^{-1}$, $\wideparen{\widetilde{g}}=([g_s, 1], k)\in \wideparen{\widetilde{\GSp}}(W)$.  Then: $(v,t)\alpha(\wideparen{\widetilde{g}})=(vg_s, \lambda_{g_s} t)$. Hence  the results holds.\\
2) It is a consequence of (1) by Clifford theory.
\end{proof}
\section{Even case}\label{even}
 In this section, we will consider   a finite field of  characteristic  $2$.
\subsection{Weil representation over a certain field}
Following \cite{GeLy}, \cite{GuHa}, let $\Q_2$ be the $2$-adic complete field from the rational number field. Let $K/\Q_2$ be an unramified field extension of degree $d$. Let $\mathfrak{O}_K$ or $\mathfrak{O}$ denote  the ring of integers of $K$. Then $2$ is one uniformizer of $K$. Let  $\mathfrak{m}^n=2^n \mathfrak{O}$. Let   $F=\mathfrak{O}/\mathfrak{m}$ be the residue field of order $2^{d}$.
 Let $(\mathscr{W}, \langle, \rangle_{\mathscr{W}})$ be a vector space over $K$ of dimension $2m$, with a symplectic basis $\{e_1, \cdots, e_m; e_1^{\ast}, \cdots, e_m^{\ast}\}$ so that $\langle e_i, e_j\rangle=0=\langle e_i^{\ast}, e_j^{\ast}\rangle$, $\langle e_i,e_j^{\ast}\rangle=\delta_{ij}$. Let $\mathcal{X}=\Span_{K}\{e_1, \cdots, e_m\}$, $\mathcal{X}^{\ast}=\Span_K\{e_1^{\ast}, \cdots, e_m^{\ast}\}$. For $w=x+x^{\ast}$, $w'=y+y^{\ast}\in \mathscr{W} $,  with $x, y\in \mathcal{X}$, $x^{\ast}, y^{\ast}\in \mathcal{X}^{\ast}$,  let us define
   $$B(w, w')=\langle x, y^{\ast}\rangle_{\mathscr{W}}.$$ Then $\langle w, w'\rangle_{\mathscr{W}}=B(w, w')-B(w', w)$. Let $H_{B}(\mathscr{W})=\mathscr{W}\times K$ denote   the corresponding Heisenberg group, defined as follows:
 $$(w, t)+(w',t')=(w+w', t+t'+B(w,w')).$$

   For $g\in \Sp(\mathscr{W})$, following \cite{Ra}, \cite{GuHa},  \cite{Ta}, we let $\Sigma_g$  be the set  of  continuous   functions $q$ from $\mathscr{W}$ to $K$, such that
 \begin{equation}\label{equiv}
 q(w+w')-q(w)-q (w')=B(wg, w'g)-B(w, w'), \quad \quad w, w'\in \mathscr{W}.
 \end{equation}

   Following \cite{Ta}, \cite{We}, let $Ps(\mathscr{W})=\{(g, q)\mid g\in \Sp(\mathscr{W}), q\in \Sigma_g\}$, which is called the linear pseudosymplectic group with respect to $\mathcal{X}, \mathcal{X}^{\ast}$. The group law is given as follows: $( g, q) ( g', q')=(gg', q'')$, where $q^{''}(w)=q(w)+q'(wg)$, for $w\in \mathscr{W}$. The group  $\Ps(\mathscr{W})$  can act on $H_B(\mathscr{W})$ as follows:
   $(w, t)\cdot(g,q) = (wg, t + q(w))$.
   \begin{lemma}\label{split}
   There exists a split exact sequence: $1\longrightarrow \mathscr{W}^{\vee} \longrightarrow \Ps(\mathscr{W}) \longrightarrow \Sp(\mathscr{W}) \longrightarrow 1$, where $\mathscr{W}^{\vee}=\Hom_{group}(\mathscr{W}, K)$.
   \end{lemma}
   \begin{proof}
   See \cite[p.351, Lmm.3.1]{Ra}, or \cite[Lmm.2.3]{Ta}.
   \end{proof}
   Moreover, by \cite[p.351, Lmm.3.1]{Ra}, \cite[Lmm.2.3]{Ta}, \cite[Sect.4]{We}, there exists an explicit group homomorphism
   $\alpha: \Sp(\mathscr{W}) \longrightarrow \Ps(\mathscr{W}); g\longmapsto (g, q_g)$, where
   \begin{equation}\label{B}
   \begin{split}
   q_g(x+x^{\ast})&=\tfrac{1}{2}\langle
xa, xb\rangle+\tfrac{1}{2}\langle x^{\ast}c, x^{\ast}d\rangle+\langle x^{\ast}c, xb\rangle\\
&=\tfrac{1}{2}B(
xa, xb)+\tfrac{1}{2}B( x^{\ast}c, x^{\ast}d)+B(x^{\ast}c, xb),
\end{split}
\end{equation} for $g=\begin{pmatrix} a& b\\c & d\end{pmatrix}\in \Sp(\mathscr{W})$.

 Let $\iota: \mathscr{W}^{\vee} \longrightarrow \Ps(\mathscr{W}); q_1 \longmapsto (1, q_1)$. Through $\iota$, $\alpha$, we can treat  $ \mathscr{W}^{\vee}$ and $\Sp(\mathscr{W})$ as subgroups of $\Ps(\mathscr{W})$. Then $\Ps(\mathscr{W}) \simeq \mathscr{W}^{\vee}\rtimes \Sp(\mathscr{W}) $.
For $g\in \Sp(\mathscr{W})$, $q_1\in \mathscr{W}^{\vee}$,  $( q_1, g)\in  \mathscr{W}^{\vee}\rtimes \Sp(\mathscr{W})$, and $( q_1, g)=(q_1,1) (0,g)=(0, g)({}^{g^{-1}}q_1, 1)$.

Let us extend the above result to the symplectic similitude  group. For $g\in \GSp(\mathscr{W})$,  we also let $\Sigma_g$  be the set  of  continuous  functions $q$ from $\mathscr{W}$ to $K$ such that
 \begin{equation}\label{equiv}
 q(w+w')-q(w')-q (w)=B(wg, w'g)-\lambda_g B(w, w'), \quad  w, w'\in \mathscr{W},
 \end{equation}
 where $\lambda: \GSp(\mathscr{W}) \longrightarrow K$,  the similitude character. Let $\GPs(\mathscr{W})=\{(g, q)\mid g\in \GSp(\mathscr{W}), q\in \Sigma_g\}$. The group law is given as follows: $( g, q) ( g', q')=(gg', q^{''})$, where $q^{''}(w)=\lambda_{g'}q(w)+q'(wg)$, for $w\in \mathscr{W}$.  The group  $\GPs(\mathscr{W})$  can act on $H_B(\mathscr{W})$ as follows: $(w, t)\cdot(g,q) = (wg, \lambda_{g}t + q(w))$.
    \begin{lemma}\label{split}
   There exists a split exact sequence: $1\longrightarrow \mathscr{W}^{\vee} \longrightarrow \GPs(\mathscr{W}) \longrightarrow \GSp(\mathscr{W}) \longrightarrow 1$, where $\mathscr{W}^{\vee}=\Hom_{group}(\mathscr{W}, K)$.
   \end{lemma}
   \begin{proof}
Note that $\GSp(\mathscr{W})\simeq \Sp(\mathscr{W}) \rtimes K^{\times}$, where $h: K^{\times } \longrightarrow \GSp(\mathscr{W}); t \longmapsto h_t= \begin{pmatrix} 1&0\\0 & t\end{pmatrix}$. Let us define $\alpha(h_t)=(h_t, q_{h_t})$, where  $q_{h_t}(x+x^{\ast})=0$, for $x+x^{\ast}\in \mathscr{W}$.   Then $(h_t, q_{h_t}) \in \Sigma_{h_t}$, and $\alpha|_{K^{\times}}$ is a group homomorphism.  For $g=\begin{pmatrix} a& b\\c & d\end{pmatrix}\in \Sp(\mathscr{W})$, $h_t^{-1}gh_t=\begin{pmatrix} a& bt\\ct^{-1} & d\end{pmatrix}=g'$. As $\alpha(g)\alpha(h_t)=(g,q_g) (h_t, q_{h_t})=(gh_t, q'')$, $\alpha(h_t)\alpha(g')=(h_t, q_{h_t})(g',q_{g'}) =(h_tg',q''')$.
For $w=x+x^{\ast}\in \mathscr{W}$,
$$q''(w)=tq_g(w)+q_{h_t}(wg)= t\tfrac{1}{2}\langle
xa, xb\rangle+t\tfrac{1}{2}\langle x^{\ast}c, x^{\ast}d\rangle+t\langle x^{\ast}c, xb\rangle,$$
$$q'''(w)=q_{h_t}(w)+q_{g'}(wh_t)=\tfrac{1}{2}\langle
xa, xbt\rangle+\tfrac{1}{2}\langle tx^{\ast}ct^{-1}, tx^{\ast}d\rangle+\langle tx^{\ast}ct^{-1}, xbt\rangle$$
$$=t\tfrac{1}{2}\langle
xa, xb\rangle+t\tfrac{1}{2}\langle x^{\ast}c, x^{\ast}d\rangle+t\langle x^{\ast}c, xb\rangle.$$
  \end{proof}
The map $\alpha: \GSp(\mathscr{W})\longrightarrow \GPs(\mathscr{W})$, is   given by the same formula as $(\ref{B})$.
\begin{lemma}
Let $H(\mathscr{W})$ be the usual Heisenberg group.  Then there exists an isomorphism:
$$\alpha_B:    \Sp(\mathscr{W})\ltimes H(\mathscr{W}) \longrightarrow  \Sp(\mathscr{W}) \ltimes H_B(\mathscr{W})   ,$$
$$ \qquad  [g,(x,x^{\ast}; t)]  \longmapsto [ \alpha(g), (x, x^{\ast};t+\tfrac{1}{2}B( x, x^{\ast}))].$$
\end{lemma}
\begin{proof}
According to \cite[Pro.2.1 and Lmm.2.3]{Ta},  $\alpha_B|_{ H(\mathscr{W})}$ and $\alpha_B|_{\Sp(\mathscr{W})}$  both are group homomorphisms. So it suffices to check that the semi-direct actions are comparable  on both sides.  For $w=(x, x^{\ast}) \in \mathscr{W}$, $g=\begin{pmatrix} a& b\\c & d\end{pmatrix}\in \Sp(\mathscr{W})$,  $\alpha(g)=(g, q_g)\in \Sigma_g$, so $q_g(2w)-2q_g(w)=B(wg,wg)-B(w,w)$; it implies that $q_g(w)=\tfrac{1}{2}B(wg,wg)-\tfrac{1}{2}B(w,w)$.  Then:
$\alpha_B([1, (w,t)])= [1, (w, t+\tfrac{1}{2}B(w,w))]$, $\alpha_B([g, 0])=[(g, q_g), 0]$,
 $$\alpha_B([1, (w,t)][g, 0])=\alpha_B([g, (wg, t)])= [(g,q_g), (wg, t+\tfrac{1}{2}B(wg, wg))],$$
 $$\alpha_B([1, (w,t)]) \alpha_B([g, 0])= [1, (w, t+\tfrac{1}{2}B(w,w))][(g, q_g), 0]=[(g, q_g), (wg,t+\tfrac{1}{2}B(w,w)+q_g(w))].$$
\end{proof}
 Let $\Psi$ be a non-trivial character of $K$. By Stone-von Neumann's theorem, there exists a unique (up to
isomorphism) irreducible  Heisenberg representation of $H_{B}(\mathscr{W})$ with central character $\Psi$.
 Let $c_{PR, \mathcal{X}^{\ast}}(-,-)$ be the Perrin-Rao's cocycle in $\Ha^2(\Sp(\mathscr{W}), \mu_8)$ with respect to $\mathscr{W}=\mathcal{X}\oplus \mathcal{X}^{\ast}$ and $\Psi$(cf. \cite{MoViWa}, \cite{Pe}, \cite{Ra}).  Let $1\longrightarrow \mu_8\longrightarrow  \overline{\Sp(\mathscr{W})} \longrightarrow \Sp(\mathscr{W})\longrightarrow 1$ be the associated  central topological extension of $\Sp(\mathscr{W})$ by $\mu_8$.
 \begin{theorem}[Weil]
 The Heisenberg representation of $H_B(\mathscr{W})$ can extend to be an irreducible representation of $\overline{\Sp(\mathscr{W})}\ltimes H_B(\mathscr{W})$.
\end{theorem}
 Let us denote this representation by $\Pi_{\Psi}$.  Note that the Heisenberg representation can also be extended to $\mathscr{W}^{\vee}\ltimes H_B(\mathscr{W})$.  Let us recall the  Schr\"odinger  model to realize it.
\subsection{Schr\"odinger  model}
Recall $\mathscr{W}=\mathcal{X}\oplus \mathcal{X}^{\ast}$.  Let $S(\mathcal{X})$ denote the Schwartz-Bruhat space over $\mathcal{X}$.  By \cite[p.138, Prop.]{BuHe}, \cite{Ra},  we let $\mu_K$ denote the self-dual Haar measure on $K$ with respect to $\Psi$. Let $\mu_{\mathcal{X}}$,  $\mu_{\mathcal{X}^{\ast}}$ be the Haar measures on $\mathcal{X}$, $\mathcal{X}^{\ast}$, given by the product  measures from $\mu_K$. The representation $\Pi_{\Psi}$ can be realized on $S(\mathcal{X})$ by the following formulas(cf. \cite{Ra}, \cite{MoViWa}, \cite{Pe}, \cite{We}):
\begin{equation}\label{representationsp1}
\Pi_{\Psi}[(y,y^{\ast}), k]f(x)=\Psi(k+B(x, y^{\ast})) f(x+y),
\end{equation}
\begin{equation}\label{representationsp2}
\Pi_{\Psi}[ \begin{pmatrix}
  1&b\\
  0 & 1
\end{pmatrix}, t]f(x)=t\Psi(\tfrac{1}{2}B(
x, xb)) f(x),
\end{equation}
\begin{equation}\label{representationsp3}
\Pi_{\Psi}[ \begin{pmatrix}
  a& 0\\
  0 &a^{\ast -1 }
\end{pmatrix},t]f(x)=t|det_{\mathcal{X}}(a)|^{\tfrac{1}{2}} f(xa),
\end{equation}
\begin{equation}\label{representationsp4}
\Pi_{\Psi}[ \omega, t]f(x)=t\int_{\mathcal{X}^{\ast}} \Psi(B( x, z^{\ast})) f(z^{\ast}\omega^{-1}) d\mu_{\mathcal{X}^{\ast}}(z^{\ast}),
\end{equation}
\begin{equation}\label{representationsp4}
\Pi_{\Psi}[ q]f(x)=\Psi(q(x)) f(x),
\end{equation}
where $x,y \in \mathcal{X},  y^{\ast}\in   \mathcal{X}^{\ast}$, $k\in K$, $t\in \mu_8$, $q\in  \mathscr{W}^{\vee}$, and $\begin{pmatrix}
  a& 0\\
  0 &a^{\ast -1 } \end{pmatrix}$, $\begin{pmatrix}
  1&b\\
  0 & 1
\end{pmatrix}$, $\omega\in \Sp( \mathscr{W})$, $e_i \omega=-e_i^{\ast}$, $e_i^{\ast}\omega=e_i$.

Moreover, under such actions, for $g_1, g_2\in \Sp( \mathscr{W})$, $\Pi_{\Psi}(g_1)\Pi_{\Psi}(g_2)=c_{PR, \mathcal{X}^{\ast}}(g_1,g_2)\Pi_{\Psi}(g_1g_2)$, for the Perrin-Rao's 2-cocycle $c_{PR, \mathcal{X}^{\ast}}(-,-)$ associated with  $\Psi$ and $\mathcal{X}^{\ast}$.

For a subset  $S\subseteq \{1, \cdots, m\}$,  define the element  $\omega_S \in \Sp(\mathscr{W})$ as follows:
$$ (e_i)\omega_S=\left\{\begin{array}{lr}
-e_i^{\ast}& i\in S\\
 e_i & i\notin S
 \end{array}\right. \textrm{ and } (e_i^{\ast})\omega_S=\left\{\begin{array}{lr}
e_i & i\in S \\
e_i^{\ast}& i\notin S
 \end{array}\right..$$ Let $\mathcal{X}_S=\Span\{ e_i\mid i\in S\}$, $\mathcal{X}_{S'}=\Span\{ e_i\mid i\notin S\}$ and $\mathcal{X}_S^{\ast}= \Span\{ e^{\ast}_i\mid i\in S\}$, $\mathcal{X}_{S'}^{\ast}= \Span\{ e^{\ast}_i\mid i\notin S\}$. By \cite[p.388, Prop.2.1.4]{Pe} or \cite[p.351, Lmm.3.2,(3.9)]{Ra}, for $x=x_s+x_{s'}\in \mathcal{X}_S\oplus \mathcal{X}_{S'}$,
 \begin{equation}
  \begin{split}
 [\Pi_{\Psi}(\omega_S)f](x_s+x_{s'})&=\int_{\mathcal{X}_S^{\ast}} \Psi(\tfrac{1}{2}\langle x_{s'} \omega_S,x_s  \omega_S\rangle+ \langle z_s^{\ast} \omega_S,  x_s\omega_S\rangle) f(x_{s'}\omega_S+z_s^{\ast}\omega_S)d\mu_{\mathcal{X}_S^{\ast}}(z_s^{\ast})\\
  &=\int_{\mathcal{X}_S^{\ast}} \Psi(B( x_s,  z_s^{\ast})) f(x_{s'}\omega_S+z_s^{\ast}\omega^{-1}_S)d\mu_{\mathcal{X}_S^{\ast}}(z_s^{\ast}).
    \end{split}
  \end{equation}
\subsection{Lattice model}
Let $\Psi$ now be a character of level  $ \mathfrak{m}^2$.  Let $$\mathcal{L}=\Span_{\mathfrak{O}_K}\{ 2e_1, \cdots, 2e_m;  2e_1^{\ast}, \cdots,  2e_m^{\ast}\},$$  which is a self-dual lattice of  $(\mathscr{W}, \langle, \rangle_{\mathscr{W}})$ with respect to $\Psi$. Let $H_{B}(\mathcal{L})=\mathcal{L}\times K$  denote   the corresponding Heisenberg subgroup of $H_{B}(\mathscr{W})$. Let $\Psi_{\mathcal{L}}$ be the extended  character of $H_{B}(\mathcal{L})$ defined as: $(l, t) \longmapsto \Psi(t)$, for $l\in \mathcal{L}$, $t \in K$.
\begin{proposition}
   $(\Pi_{\Psi}=\cInd_{H_{B}(\mathcal{L})}^{H_{B}(\mathscr{W})} \Psi_A,  S_{\mathcal{L}}=\cInd_{H_{B}(\mathcal{L})}^{H_{B}(\mathscr{W})} \C)$ defines a Heisenberg representation of $H_{B}(\mathscr{W})$ with central character $\Psi$.
\end{proposition}
\begin{proof}
See \cite[pp.28-30]{MoViWa}, \cite[pp.9-11]{Bl}, \cite[pp.8-9]{Ta}.
\end{proof}
\begin{remark}
Let us go back to  the ordinary Heisenberg group $H(\mathscr{W})$. Let $H(\mathcal{L})=\mathcal{L}\times K\subseteq H(\mathscr{W})$. The character $\Psi$ of $K$ can extend to  $H(\mathcal{L})$ by defining $\Psi_{\mathcal{L}}(l, t)=\Psi(t-\frac{1}{2}\langle x_l, x_l^{\ast}\rangle)$, for $l=x_l+x_{l}^{\ast}\in \mathcal{L}=\mathcal{L}\cap \mathcal{X} +\mathcal{L}\cap \mathcal{X}^{\ast}$. According to \cite[I.3]{MoViWa}, $\cInd_{H(\mathcal{L})}^{H(\mathscr{W})}\Psi_{\mathcal{L}}$ defines a Heisenberg representation $H(\mathscr{W})$ with the central character $\Psi$.
\end{remark}
 The space $S_{\mathcal{L}}$ consists  of locally constant, compactly supported functions $f$ on $\mathscr{W}$  such that
 $$f(l+w)=\Psi(-B( l, w))f(w),$$ for $l\in \mathcal{L}$, $w\in \mathscr{W}$.  For $h=(w',t)\in H_{B}(\mathscr{W})$,
\begin{equation}\label{BB'}
\Pi_{\psi}(h) f(w)=f([w,0]+[w',t])=f([w+w', t+B(w,w')])=\Psi( t+B(w, w')) f(w+w').
\end{equation}

  Following \cite[p.138, Prop.]{BuHe}, let $\mu_K$ denote the Haar measure on $K$ such that $\mu_K(\mathfrak{O})=|F|$. Let $\mu_{\mathcal{X}}$, $\mu_{\mathcal{X}^{\ast}}$ be the product measures on $\mathcal{X}$ and $\mathcal{X}^{\ast}$ from $\mu_K$. Let $\mu_{\mathcal{X}\cap \mathcal{L}}= \mu_{\mathcal{X}}|_{\mathcal{X}\cap \mathcal{L}}$, and $\mu_{\mathcal{X}^{\ast} \cap\mathcal{L}}=\mu_{\mathcal{X}^{\ast}}|_{\mathcal{X}^{\ast}\cap \mathcal{L}}$.  Let $\mu_{\mathcal{X}/\mathcal{X} \cap\mathcal{L}}$, resp. $\mu_{\mathcal{X}^{\ast}/\mathcal{X}^{\ast} \cap\mathcal{L}}$  denote the  Haar measures  on $\mathcal{X}/\mathcal{X} \cap\mathcal{L}$, resp. $\mathcal{X}^{\ast}/\mathcal{X}^{\ast} \cap\mathcal{L}$,  such that $\mu_{\mathcal{X}}=\mu_{\mathcal{X} \cap\mathcal{L}} \mu_{\mathcal{X}/\mathcal{X} \cap\mathcal{L}}$, $\mu_{\mathcal{X}^{\ast}}=\mu_{\mathcal{X}^{\ast} \cap\mathcal{L}} \mu_{\mathcal{X}^{\ast}/\mathcal{X}^{\ast} \cap\mathcal{L}}$. For simplicity of notations, sometimes we  omit the notation $\mu$ and  those subscripts.

   By \cite[pp.164-165]{We}, there exists  a pair of explicit isomorphisms between   $S(\mathcal{X})$ and  $S_{\mathcal{L}}$, given as follows:
 \begin{equation}\label{eq1}
 \mathcal{F}_{\mathcal{X}\mathcal{L}}(f')(w)=\int_{\mathcal{L}\cap \mathcal{X}} f'(x+l) \Psi( B( l, w)) d l,
 \end{equation}
   \begin{equation}\label{eq2}
  \mathcal{F}_{\mathcal{L}\mathcal{X}}(f)(y)=  \int_{\mathcal{X}^{\ast}/\mathcal{X}^{\ast}\cap \mathcal{L}} f(y+\dot{y}^{\ast}) d\dot{y}^{\ast},
      \end{equation}
      for $ w=x+ x^{\ast} \in \mathscr{W}$, $y\in \mathcal{X}$,  $f'\in S(\mathcal{X})$,  $f\in S_{\mathcal{L}}$.

     For an element   $f\in S_{\mathcal{L}}$, let $f'=\mathcal{F}_{\mathcal{L}\mathcal{X}}(f)\in S(\mathcal{X})$. Then $f=\mathcal{F}_{\mathcal{X}\mathcal{L}}(f')$.  Through  the isomorphisms $\mathcal{F}_{\mathcal{X}\mathcal{L}}$, $\mathcal{F}_{\mathcal{L}\mathcal{X}}$, for $g\in  \Sp(\mathscr{W})$, we can define
     \begin{equation}
     \Pi_{\Psi}(g)f(w)=\mathcal{F}_{\mathcal{X}\mathcal{L}}[\Pi_{\Psi}(g)(f')](w)=\mathcal{F}_{\mathcal{X}\mathcal{L}}[\Pi_{\Psi}(g)\mathcal{F}_{\mathcal{L}\mathcal{X}}(f)](w).
     \end{equation}
     Under such action, for $g_1, g_2\in \Sp(\mathscr{W})$,
     \begin{equation}
     \begin{split}
     \Pi_{\Psi}(g_1)[\Pi_{\Psi}(g_2)f]&=\mathcal{F}_{\mathcal{X}\mathcal{L}}[\Pi_{\Psi}(g_1)\mathcal{F}_{\mathcal{L}\mathcal{X}}]([\Pi_{\Psi}(g_2)f])\\
     &=\mathcal{F}_{\mathcal{X}\mathcal{L}}[\Pi_{\Psi}(g_1)\mathcal{F}_{\mathcal{L}\mathcal{X}}]\mathcal{F}_{\mathcal{X}\mathcal{L}}
     [\Pi_{\Psi}(g_2)\mathcal{F}_{\mathcal{L}\mathcal{X}}](f)\\
     &=\mathcal{F}_{\mathcal{X}\mathcal{L}}[\Pi_{\Psi}(g_1)\Pi_{\Psi}(g_2)\mathcal{F}_{\mathcal{L}\mathcal{X}}](f)\\
     &=c_{PR, \mathcal{X}^{\ast}}(g_1, g_2)\Pi_{\Psi}(g_1g_2)f.
     \end{split}
     \end{equation}
\subsection{ $\Sp(\mathcal{L})$} In particular, let us consider  $g\in \Sp(\mathcal{L})$. In this case,  $\mathcal{L}g = \mathcal{L}$.  Under the basis $\{ 2e_1, \cdots, 2e_m;  2e_1^{\ast}, \cdots,  2e_m^{\ast}\}$ of $\mathcal{L}$, $\Sp(\mathcal{L})\simeq \Sp_{2m}(\mathfrak{O})$.  Let $w=x+x^{\ast} \in \mathscr{W}$. \\
Case 1: $g=\begin{pmatrix}
  1&b\\
  0 & 1
\end{pmatrix} \in \Sp_{2m}(\mathfrak{O})$, and $b\in  M_m(\mathfrak{m})$. In this case, $\Psi(q_g(l))=1$, for any $l\in\mathcal{L}\cap \mathcal{X}$.
\begin{equation}
\begin{split}
\Pi_{\Psi}(g)f(w)& =\int_{\mathcal{L}\cap \mathcal{X}} [\Pi_{\Psi}(g)f'](x+l) \Psi( B( l, w)) dl \\
& =\int_{\mathcal{L}\cap \mathcal{X}} \Psi(\frac{1}{2}B(x+l, (x+l)b)) f'(x+l) \Psi( B( l, w)) d l\\
&=\int_{\mathcal{L}\cap \mathcal{X}} \int_{\mathcal{X}^{\ast}/\mathcal{X}^{\ast} \cap \mathcal{L}} \Psi(\frac{1}{2}B(x+l, (x+l)b)) f(x+l+ \dot{y}^{\ast}) \Psi( B( l, w)) d\dot{y}^{\ast}dl\\
& =\Psi(q_g(w))\int_{\mathcal{L}\cap \mathcal{X}} \int_{\mathcal{X}^{\ast}/\mathcal{X}^{\ast} \cap \mathcal{L}} f(x+l+ \dot{y}^{\ast}) \Psi( B( l, x^{\ast}+xb)) d\dot{y}^{\ast}dl\\
&=\Psi(q_g(w)) f(wg).
\end{split}
\end{equation}
Case 2: $g=h_a=\begin{pmatrix}
  a&\\
   & (a^{\ast})^{-1}
\end{pmatrix} \in \Sp(\mathcal{L})$.
\begin{equation}
\begin{split}
\Pi_{\Psi}(g)f(w)& =\int_{\mathcal{L}\cap \mathcal{X}} [\Pi_{\Psi}(g)f'](x+l) \Psi( B( l, w)) dl\\
&=\int_{\mathcal{L}\cap \mathcal{X}} |det_{\mathcal{X}}(a)|^{\tfrac{1}{2}}  f'((x+l)a) \Psi( B( l, w)) d l\\
&=\int_{\mathcal{L}\cap \mathcal{X}} |det_{\mathcal{X}}(a)|^{\tfrac{1}{2}} \int_{\mathcal{X}^{\ast}/\mathcal{X}^{\ast}\cap \mathcal{L}} f((x+l)a+\dot{y}^{\ast}) \Psi( B( l, w)) d l  d\dot{y}^{\ast}\\
&=\int_{\mathcal{L}\cap \mathcal{X}}\int_{\mathcal{X}^{\ast}/\mathcal{X}^{\ast}\cap \mathcal{L}}  f((x+l)a+\dot{y}^{\ast})\Psi( B( l, w)) d(la)d\dot{y}^{\ast}\\
&=\int_{\mathcal{L}\cap \mathcal{X}}\int_{\mathcal{X}^{\ast}/\mathcal{X}^{\ast}\cap \mathcal{L}}  f((xa+l'+\dot{y}^{\ast})\Psi( B( l'a^{-1}, w)) dl'd\dot{y}^{\ast}\\
&=f(xa+x^{\ast}(a^{\ast})^{-1})=f(wg)=\Psi(q_g(w)) f(wg).
\end{split}
\end{equation}
Case 3: $g=\omega_S \in \Sp(\mathcal{L})$.  Let $w=x+ x^{\ast}$, for $x=x_s+x_{s'}\in \mathcal{X}_S\oplus\mathcal{X}_{S'} $, $x^{\ast}=x^{\ast}_s+x^{\ast}_{s'}\in \mathcal{X}^{\ast}_S\oplus\mathcal{X}^{\ast}_{S'} $.
\begin{equation}\label{eq3}
\begin{split}
\Pi_{\Psi}(g)f(w)& =\int_{\mathcal{L}\cap \mathcal{X}} [\Pi_{\Psi}(g)f'](x+l) \Psi( B( l, w)) dl\\
&= \int_{\mathcal{L}\cap \mathcal{X}}  \int_{\mathcal{X}_S^{\ast}} \Psi( B( l, w))\Psi(B( x_s+l_s,  z_s^{\ast})) f'((x_{s'}+l_{s'})\omega_S+z_s^{\ast}\omega^{-1}_S)dz_s^{\ast}dl\\
&=\int_{\mathcal{L}\cap \mathcal{X}}  \int_{\mathcal{X}_S^{\ast}}   \int_{\mathcal{X}^{\ast}/\mathcal{X}^{\ast}\cap \mathcal{L}} \Psi( B( l, w))\Psi(B( x_s+l_s,  z_s^{\ast})) f((x_{s'}+l_{s'})\omega_S+z_s^{\ast}\omega^{-1}_S+\dot{y}^{\ast}) d\dot{y}^{\ast} dz_s^{\ast}   dl\\
&= \int_{\mathcal{X}_S^{\ast}}   \int_{\mathcal{X}^{\ast}/\mathcal{X}^{\ast}\cap \mathcal{L}} f(x_{s'}\omega_S+z_s^{\ast}\omega^{-1}_S+\dot{y}^{\ast})\Psi(B( x_s,  z_s^{\ast})) \\ &\qquad\qquad\qquad\qquad\qquad\cdot[\int_{\mathcal{L}\cap \mathcal{X}}\Psi(-B(l_{s'}\omega_S, \dot{y}^{\ast})  \Psi( B( l, w))\Psi(B( l_s,  z_s^{\ast}))dl] d\dot{y}^{\ast}dz_s^{\ast}.
\end{split}
\end{equation}
\begin{equation}
\begin{split}
&\int_{\mathcal{L}\cap \mathcal{X}}\Psi(-B(l_{s'}\omega_S, \dot{y}^{\ast})  \Psi( B( l, w))\Psi(B( l_s,  z_s^{\ast}))dl\\
&=\int_{\mathcal{L}\cap \mathcal{X}_S}  \Psi( B( l_s, x_{s}^{\ast}+ z_s^{\ast}))dl_s   \int_{\mathcal{L}\cap \mathcal{X}_{S'}}  \Psi(B(l_{s'}, -\dot{y}_{s'}^{\ast}+x_{s'}^{\ast}) dl_{s'},
\end{split}
\end{equation}
which is non-zero iff $x_s^{\ast}+z_s^{\ast}\in \mathcal{L}\cap X_S^{\ast}$, and $x_{s'}^{\ast}-\dot{y}_{s'}^{\ast} \in \mathcal{L}\cap X_{S'}^{\ast}$. Moreover, $\int_{\mathcal{L}\cap \mathcal{X}_S}  dl_s=   \int_{\mathcal{L}\cap \mathcal{X}_{S'}} dl_{s'}=1$.
\begin{gather}
\begin{split}
(\ref{eq3})&=\int_{\mathcal{L}\cap \mathcal{X}_S^{\ast}}   \int_{\mathcal{X}_S^{\ast}/\mathcal{X}_S^{\ast}\cap \mathcal{L}}    f(x_{s'}\omega_S+(-x_{s}^{\ast}+z_s^{\ast})\omega^{-1}_S+\dot{y}_s^{\ast}+ x_{s'}^{\ast})\Psi(B( x_s,  -x_{s}^{\ast}+z_s^{\ast}))dz_s^{\ast}d\dot{y}_s^{\ast}\\
&=\Psi(B( x_s,  -x_{s}^{\ast}))\int_{\mathcal{L}\cap \mathcal{X}_S^{\ast}}   \int_{\mathcal{X}_S^{\ast}/\mathcal{X}_S^{\ast}\cap \mathcal{L}}    f(x_{s'}\omega_S+x_{s}^{\ast}\omega_S+z_s^{\ast}\omega^{-1}_S+\dot{y}_s^{\ast}+ x_{s'}^{\ast}\omega_S)\Psi(B( x_s,  z_s^{\ast}))dz_s^{\ast}d\dot{y}_s^{\ast}\\
&=\Psi(B( x_s,  -x_{s}^{\ast}))\int_{\mathcal{L}\cap \mathcal{X}^{\ast}_S}   \int_{\mathcal{X}_S^{\ast}/\mathcal{X}_S^{\ast}\cap \mathcal{L}}    f(x_{s'}\omega_S+z_s^{\ast}\omega^{-1}_S+\dot{y}_s^{\ast}+x^{\ast}\omega_S)\Psi(B( x_s ,  z_s^{\ast}))dz_s^{\ast}d\dot{y}_s^{\ast}\\
&=\Psi(B( x_s,  -x_{s}^{\ast}))\int_{\mathcal{L}\cap \mathcal{X}_S}   \int_{\mathcal{X}_S^{\ast}/\mathcal{X}_S^{\ast}\cap \mathcal{L}}    f(x_{s'}\omega_S+l_s+\dot{y}_s^{\ast}+x^{\ast}\omega_S) \Psi(B(l_s, x_s \omega_S ))dl_sd\dot{y}_s^{\ast}\\
&=\Psi(B( x_s,  -x_{s}^{\ast}))f(w\omega_S )=\Psi(q_g(w)) f(wg).
\end{split}
\end{gather}

Case 4: $g=q \in \mathscr{W}^{\vee}$, and $\Psi(q(l))=1$, for $l\in \mathcal{X}\cap \mathcal{L}$.
\begin{equation}
\begin{split}
\Pi_{\Psi}(g)f(w)& =\int_{\mathcal{L}\cap \mathcal{X}} [\Pi_{\Psi}(g)f'](x+l) \Psi( B( l, w)) dl\\
&=\int_{\mathcal{L}\cap \mathcal{X}} \Psi(q(x+l))  f'(x+l) \Psi( B( l, w)) d l\\
&=\int_{\mathcal{L}\cap \mathcal{X}} \int_{\mathcal{X}^{\ast}/\mathcal{X}^{\ast}\cap \mathcal{L}} \Psi(q(x+l)) f(x+l+\dot{y}^{\ast}) \Psi( B( l, w)) d l  d\dot{y}^{\ast}\\
&=\int_{\mathcal{L}\cap \mathcal{X}}\int_{\mathcal{X}^{\ast}/\mathcal{X}^{\ast}\cap \mathcal{L}} \Psi(q(x)) f(x+l+\dot{y}^{\ast})\Psi( B( l, w)) dld\dot{y}^{\ast}\\
&=\Psi(q(x)) f(w).
\end{split}
\end{equation}
\subsubsection{Iwahori decomposition}\label{Iwode} Let us recall some results from connected reductive groups.  By \cite[p.54]{MoViWa}),  $\Sp(\mathscr{W})=\cup_{S} P(\mathcal{X}^{\ast})\omega_SP(\mathcal{X}^{\ast})$.

Under the basis $\{e_1, \cdots, e_m;e_1^{\ast}, \cdots, e_m^{\ast}\}$ of $\mathscr{W}$, we identity $\Sp(\mathscr{W})$ with $\Sp_{2m}(K)$. Let $B(K)$ denote the standard Borel subgroup of   $\Sp_{2m}(K)$(cf. \cite[p.104]{Sp}), $N(K)$ the unipotent subgroup of $B(K)$, and let $N^-(K)$ denote  the opposite group of $N(K)$.   Let $I$ be an Iwahori subgroup of $\Sp_{2m}(\mathfrak{O})$, which is the inverse image of $B(\mathfrak{O}/\mathfrak{m})$ in $\Sp_{2m}(\mathfrak{O})$.  Let $N=N(K)\cap \Sp_{2m}(\mathfrak{O})$, $B=B(K)\cap \Sp_{2m}(\mathfrak{O})$, and $N^-=N^-(K)\cap \Sp_{2m}(\mathfrak{O})$. Let $\mathfrak{W}$ denote the usual  Weyl group of $\Sp_{2m}(F)$ such that $\omega_S\in \mathfrak{W}$ and the symmetric group  $S_m \subseteq \mathfrak{W}$.

According to \cite[p.216]{CaSh},\cite{DeKaVi},\cite{Vi}, $\Sp_{2m}(\mathfrak{O})=\cup_{\tau \in \mathfrak{W}} I\tau I$, and $I=(N^-\cap I)(B\cap I)$.  Let $\Gamma_1=\{ g\in \Sp_{2m}( \mathfrak{O}) \mid  g\equiv I (\mod\mathfrak{m})\}$. Then $\Gamma_1\subseteq I$, $(N^-\cap I) \subseteq \Gamma_1$, and $\Gamma_1=(N^-\cap \Gamma_1)(B\cap \Gamma_1)$.
Recall that if  $S=\{1, \cdots, m\}$, we write simply $\omega$ instead of $\omega_S$.
\begin{lemma}\label{coco}
$c_{PR, \mathcal{X}^{\ast}}(\omega,\omega^{-1}g)\Pi_{\Psi}(g)f(w)=\Psi(q_g(w))f(wg)$, for $g\in N^-\cap \Gamma_1$.
\end{lemma}
\begin{proof}
For $g\in N^-\cap \Gamma_1$, $g=\omega g_1\omega^{-1}$, for some $g_1\in N\cap \Gamma_1 \subseteq P(\mathcal{X}^{\ast})$. Then:
\begin{equation}\label{eqcc}
\begin{split}
c_{PR, \mathcal{X}^{\ast}}(\omega,g_1\omega^{-1})\Pi_{\Psi}(g)f(w)&=c_{PR, \mathcal{X}^{\ast}}(\omega,g_1\omega^{-1})\Pi_{\Psi}(\omega g_1\omega^{-1})f(w)=\Pi_{\Psi}(\omega)\Pi_{\Psi}( g_1\omega^{-1})f(w)\\
&=\Pi_{\Psi}(\omega)\Pi_{\Psi}( g_1) \Pi_{\Psi}(\omega^{-1})f(w)=\Psi(q_g(w))f(wg).
\end{split}
\end{equation}
\end{proof}
\begin{remark}
  $c_{PR, \mathcal{X}^{\ast}}(\omega,\omega^{-1}g)=c_{PR, \mathcal{X}^{\ast}}(\omega,g)^{-1}$.
\end{remark}
\begin{proof}
Note that $\omega^{-1}=-\omega$. By \cite[p.359]{Ra}, $c_{PR, \mathcal{X}^{\ast}}(\omega,\omega^{-1})=1=c_{PR, \mathcal{X}^{\ast}}(\omega,\omega)$. By the equality:
$c_{PR, \mathcal{X}^{\ast}}(\omega,\omega^{-1}g)c_{PR, \mathcal{X}^{\ast}}(\omega^{-1},g)=c_{PR, \mathcal{X}^{\ast}}(\omega,\omega^{-1})c_{PR, \mathcal{X}^{\ast}}(1,g)=1$, we obtain $c_{PR, \mathcal{X}^{\ast}}(\omega,\omega^{-1}g)=c_{PR, \mathcal{X}^{\ast}}(\omega^{-1},g)^{-1}=c_{PR, \mathcal{X}^{\ast}}(\omega,g)^{-1}$.
\end{proof}
For any $g\in \Gamma_1$, let us write $g=g_1g_2$, for some unique $g_1\in N^-\cap \Gamma_1$,  $g_2\in B\cap \Gamma_1$. Let us define a function
$$t: \Gamma_1 \longrightarrow \mu_8; g \longmapsto c_{PR, \mathcal{X}^{\ast}}(\omega,\omega^{-1}g_1)=c_{PR, \mathcal{X}^{\ast}}(\omega,\omega^{-1}g).$$  Let us modify the action of $\Gamma_1$  by defining
$$\Pi_{\Psi}^{\sim}(g)f=\Pi_{\Psi}(g) t(g)f, \qquad f\in S_{\mathcal{L}}.$$
Let $c'_{PR, \mathcal{X}^{\ast}}(g,g')=c_{PR, \mathcal{X}^{\ast}} (g,g') t(g)t(g') t(gg')^{-1}$, for $g,g'\in \Gamma_1$. It is clear that $\Pi_{\Psi}^{\sim}(g)\Pi_{\Psi}^{\sim}(g')=c'_{PR, \mathcal{X}^{\ast}}(g,g') \Pi_{\Psi}^{\sim}(gg')$.
\begin{lemma}\label{Gamma}
\begin{itemize}
\item[(1)] $\Pi_{\Psi}^{\sim}(g)f(w)=\Psi(q_g(w))f(wg)$, for $g\in \Gamma_1$.
\item[(2)] $c'_{PR, \mathcal{X}^{\ast}}|_{\Gamma_1\times \Gamma_1}=1$.
\end{itemize}
\end{lemma}
\begin{proof}
By the above formulas, for an element $g\in \Gamma_1$, $g=g_1g_2$, for some $g_1\in N^-\cap \Gamma_1$, $g_2\in B\cap \Gamma_1$. So
$\Pi_{\Psi}(g)f(w)=c_{PR, \mathcal{X}^{\ast}}(g_1,g_2) \Pi_{\Psi}(g_1)\Pi_{\Psi}(g_2)f(w)=\Pi_{\Psi}(g_1)\Pi_{\Psi}(g_2) f(w)=c_{PR, \mathcal{X}^{\ast}}(\omega,\omega^{-1}g_1)^{-1}\Psi(q_g(w))f(wg)=c_{PR, \mathcal{X}^{\ast}}(\omega,\omega^{-1}g)^{-1}\Psi(q_g(w))f(wg)$. Hence $\Pi_{\Psi}^{\sim}(g)f(w)=\Psi(q_g(w))f(wg)$. Note that $(2)$ is a consequence of (1).
\end{proof}

\subsection{ $\Sp(\mathcal{L}^{\tfrac{1}{2}})$.}
Let $\mathcal{L}^{\tfrac{1}{2}}=\Span_{\mathfrak{O}}\{ e_1, \cdots, e_m; 2e_1^{\ast}, \cdots, 2e_m^{\ast}\}=\tfrac{1}{2}\mathcal{L}\cap \mathcal{X}\oplus \mathcal{X}^{\ast}\cap \mathcal{L}$. Let $V_{1}$ be the subspace  of functions   supported on $\mathcal{L}^{\tfrac{1}{2}}$. Under the basis $\{ e_1,\cdots, e_m; e_1^{\ast},\cdots, e_m^{\ast}\}$ of $\mathscr{W}$, $\Sp(  \mathcal{L}^{\tfrac{1}{2}})\simeq \begin{pmatrix}
  1&0\\
  0 & \tfrac{1}{2}
\end{pmatrix}\Sp_{2m}(\mathfrak{O}) \begin{pmatrix}
  1&0\\
  0 &2
\end{pmatrix}$, which can be generated by $h_{a}=\begin{pmatrix}
  a& 0\\
   0& (a^{\ast})^{-1}\end{pmatrix}$, $u(b)=\begin{pmatrix}
  1&b\\
   0& 1
\end{pmatrix} $, $h_{(S,1/2)}\omega_S$, where   $ h_{(S, 1/2)}=\begin{pmatrix}
  a_2&\\
   & (a_2^{\ast})^{-1}
\end{pmatrix} \in  \Sp(\mathscr{W})$, with $a_2= \begin{pmatrix}
   21_S&\\
   & 1_{S'}
\end{pmatrix} $, $(a_2^{\ast})^{-1}=\begin{pmatrix}
  \tfrac{1}{2}1_S&\\
   &1_{S'}
\end{pmatrix} $, $S\subseteq \{1,\cdots, m\}$, $S'=\{1,\cdots, m\} \setminus S$; $(e_i) 21_S=2e_i$, $(e_i^{\ast}) \tfrac{1}{2}1_S=\tfrac{1}{2}e_i^{\ast}$, for $i\in S$. Let $f\in V_{1}$. By the formulas in the above cases 1 and 2, $\Pi_{\Psi}(g)f(w)=\Psi(q_g(w)) f(wg)$, for $g=h_a$, $u(b)\in \Sp(  \mathcal{L}^{\tfrac{1}{2}})$.

Case 5: $g=h_{(S,1/2)}\omega_S \in \Sp(\mathcal{L}^{\tfrac{1}{2}})$.
\begin{equation}\label{eq5}
\begin{split}
\Pi_{\Psi}(g)f(w)&
=\int_{\mathcal{L}\cap \mathcal{X}} [\Pi_{\Psi}(g)f'](x+l) \Psi( B( l, w)) dl\\
&=\int_{\mathcal{L}\cap \mathcal{X}} |det_{\mathcal{X}}(a_2)|^{\tfrac{1}{2}} [\Pi_{\Psi}(\omega_S)f']((x+l)a_2) \Psi( B( l, w)) dl\\
&=\int_{\mathcal{L}\cap \mathcal{X}}\int_{\mathcal{X}_S^{\ast}} |det_{\mathcal{X}}(a_2)|^{\tfrac{1}{2}} \Psi(B( 2x_s+2l_s,  z_s^{\ast})) f'(x_{s'}\omega_S+l_{s'}\omega_S+z_s^{\ast}\omega^{-1}_S) \Psi( B( l, w))dz_s^{\ast} dl\\
&= \int_{\mathcal{X}_S^{\ast}}   \int_{\mathcal{X}^{\ast}/\mathcal{X}^{\ast}\cap \mathcal{L}} |det_{\mathcal{X}}(a_2)|^{\tfrac{1}{2}}f(x_{s'}\omega_S+z_s^{\ast}\omega^{-1}_S+\dot{y}^{\ast})\Psi(B( 2x_s,  z_s^{\ast})) \\ &\qquad\qquad\qquad\qquad\qquad\cdot[\int_{\mathcal{L}\cap \mathcal{X}}\Psi(-B(l_{s'}\omega_S, \dot{y}^{\ast}))  \Psi( B( l, w))\Psi(B( 2l_s,  z_s^{\ast}))dl] d\dot{y}^{\ast} dz_s^{\ast}.
\end{split}
\end{equation}
\begin{equation}
\begin{split}\label{eqq61}
&\int_{\mathcal{L}\cap \mathcal{X}}\Psi(-B(l_{s'}\omega_S, \dot{y}^{\ast}))  \Psi( B( l, w))\Psi(B(2l_s,  z_s^{\ast}))dl\\
&=\int_{\mathcal{L}\cap \mathcal{X}_S}  \Psi( B( l_s, x_{s}^{\ast}+ 2z_s^{\ast}))dl_s   \int_{\mathcal{L}\cap \mathcal{X}_{S'}}  \Psi(B(l_{s'}, -\dot{y}_{s'}^{\ast}+x_{s'}^{\ast}) dl_{s'},
\end{split}
\end{equation}
which is non-zero iff $x_s^{\ast}+2z_s^{\ast}\in \mathcal{L}\cap  \mathcal{X}_S^{\ast}$, and $x_{s'}^{\ast}-\dot{y}_{s'}^{\ast} \in \mathcal{L}\cap  \mathcal{X}_{S'}^{\ast}$. Moreover, $\int_{\mathcal{L}\cap \mathcal{X}_S}  dl_s=\int_{\mathcal{L}\cap \mathcal{X}_{S'}} dl_{s'}=1$.
\begin{gather}
\begin{split}
(\ref{eqq61})&=\int_{\tfrac{1}{2}\mathcal{L}\cap \mathcal{X}_S^{\ast}}  \int_{\mathcal{X}_S^{\ast}/\mathcal{X}_S^{\ast}\cap \mathcal{L}}    |F|^{-\frac{|S|}{2}}f(x_{s'}\omega_S+(-\tfrac{1}{2}x_{s}^{\ast}+z_s^{\ast})\omega^{-1}_S+\dot{y}_s^{\ast}+ x_{s'}^{\ast})\Psi(B(2x_s,  -\tfrac{1}{2}x_{s}^{\ast}+z_s^{\ast}))dz_s^{\ast}d\dot{y}_s^{\ast}\\
&=\Psi(B( 2x_s,  -\tfrac{1}{2}x_{s}^{\ast}))\int_{\mathcal{L}\cap \mathcal{X}_S^{\ast}}   \int_{\mathcal{X}_S^{\ast}/\mathcal{X}_S^{\ast}\cap \mathcal{L}}    |F|^{\frac{|S|}{2}}f(x_{s'}\omega_S+\tfrac{1}{2}x_{s}^{\ast}\omega_S+\tfrac{1}{2}z_s^{\ast}\omega^{-1}_S+\dot{y}_s^{\ast}+ x_{s'}^{\ast}\omega_S)\Psi(B(2 x_s,  \tfrac{1}{2}z_s^{\ast}))dz_s^{\ast}d\dot{y}_s^{\ast}\\
&=\Psi(B( 2x_s,  -\tfrac{1}{2}x_{s}^{\ast}))\int_{\mathcal{L}\cap \mathcal{X}^{\ast}_S}   \int_{\mathcal{X}_S^{\ast}/\mathcal{X}_S^{\ast}\cap \mathcal{L}}    |F|^{\frac{|S|}{2}}f(x_{s'}\omega_S+\tfrac{1}{2}z_s^{\ast}\omega^{-1}_S+\dot{y}_s^{\ast}+x^{\ast}g)\Psi(B(2 x_s ,  \tfrac{1}{2}z_s^{\ast}))dz_s^{\ast}d\dot{y}_s^{\ast}\\
&=\Psi(B( 2x_s,  -\tfrac{1}{2}x_{s}^{\ast}))\int_{\mathcal{L}\cap \mathcal{X}_S}   \int_{\mathcal{X}_S^{\ast}/\mathcal{X}_S^{\ast}\cap \mathcal{L}}    |F|^{\frac{|S|}{2}}f(x_{s'}\omega_S+\tfrac{1}{2}l_s+\dot{y}_s^{\ast}+x^{\ast}g) \Psi(B(\tfrac{1}{2}l_s, 2x_s \omega_S ))dl_sd\dot{y}_s^{\ast}\\
&=\Psi(B( 2x_s,  -\tfrac{1}{2}x_{s}^{\ast}))\int_{\tfrac{1}{2}\mathcal{L}\cap \mathcal{X}_S}   \int_{\mathcal{X}_S^{\ast}/\mathcal{X}_S^{\ast}\cap \mathcal{L}}    f(x_{s'}\omega_S+l_s+\dot{y}_s^{\ast}+x^{\ast}g) \Psi(B(l_s, 2x_s \omega_S ))dl_sd\dot{y}_s^{\ast}
\end{split}
\end{gather}
\begin{gather}
\begin{split}
&=\Psi(B( 2x_s,  -\tfrac{1}{2}x_{s}^{\ast}))\sum_{x_i\in \frac{1}{2}\mathcal{L}\cap \mathcal{X}_{S}/\mathcal{L}\cap \mathcal{X}_{S}}\int_{\mathcal{L}\cap \mathcal{X}_S}   \int_{\mathcal{X}_S^{\ast}/\mathcal{X}_S^{\ast}\cap \mathcal{L}}    f(x_{s'}\omega_S+x_i+l_s+\dot{y}_s^{\ast}+x^{\ast}g) \\
&\qquad\qquad\qquad\qquad\qquad\qquad\qquad\qquad\qquad\qquad\qquad\qquad\qquad\qquad\qquad\cdot\Psi(B(x_i+l_s, 2x_s \omega_S ))dl_sd\dot{y}_s^{\ast}\\
&=\Psi(B( 2x_s,  -\tfrac{1}{2}x_{s}^{\ast}))\sum_{x_i\in \frac{1}{2}\mathcal{L}\cap \mathcal{X}_{S}/\mathcal{L}\cap \mathcal{X}_{S}}\Psi(B(x_i, 2x_s \omega_S ))\int_{\mathcal{L}\cap \mathcal{X}_S}   \int_{\mathcal{X}_S^{\ast}/\mathcal{X}_S^{\ast}\cap \mathcal{L}}    f(x_{s'}\omega_S+x_i+l_s+\dot{y}_s^{\ast}+x^{\ast}g) \\
&\qquad\qquad\qquad\qquad\qquad\qquad\qquad\qquad\qquad\qquad\qquad\qquad\qquad\qquad\qquad\cdot\Psi(B(l_s, 2x_s \omega_S ))dl_sd\dot{y}_s^{\ast}\\
&=\sum_{x_i\in \frac{1}{2}\mathcal{L}\cap \mathcal{X}_{S}/\mathcal{L}\cap \mathcal{X}_{S}}\Psi(B(x_i, 2x_s \omega_S ))   \Psi(q_g(w))f(x_i+wg).
\end{split}
\end{gather}
By the above formulas, we have:
\begin{lemma}\label{cc2}
$V_{1}$ is $\Sp( \mathcal{L}^{\tfrac{1}{2}})$-stable.
\end{lemma}
\subsection{Modifying by $2$}\label{Modifyingby}
Let $(\mathscr{W}, 2\langle, \rangle)$ be another symplectic vector space over $K$.  Let $\langle, \rangle'= 2\langle, \rangle$, and  $B'(-,-)=2B(-,-)$.  For $g\in \Sp(\mathscr{W})$, let $\Sigma'_g$  be the set  of continuous functions $q'$ from $\mathscr{W}$ to $K$ such that
 \begin{equation}\label{equiv2}
 q'(w+w')-q'(w)-q' (w')=B'(wg, w'g)-B'(w, w')=2B(wg,  w'g)-2B(w, w'), \quad  w, w'\in \mathscr{W}.
 \end{equation}
 Let $P's(\mathscr{W})=\{(g, q')\mid g\in \Sp(\mathscr{W}), q'\in \Sigma'_g\}$. Let $e'_i=e_i$, $e_i^{\ast'}=\frac{1}{2} e_i^{\ast}$. Then $\{ e_1', \cdots, e'_m;  e_1^{\ast'}, \cdots,  e_m^{\ast'}\}$ forms a symplectic basis of $(W, \langle, \rangle')$.  Note that $\mathcal{X}'=\Span_{K}\{e_1', \cdots, e_m'\}=\mathcal{X}$, $\mathcal{X}^{\ast'}=\Span_K\{e_1^{\ast'}, \cdots, e_m^{\ast'}\}=\mathcal{X}^{\ast}$.

  Let $\mathcal{L}^{'}=\Span_{\mathfrak{O}_K}\{ 2e'_1, \cdots, 2e'_m;  2e_1^{\ast'}, \cdots,  2e_m^{\ast'}\}=\Span_{\mathfrak{O}_K}\{ 2e_1, \cdots, 2e_m;  e_1^{\ast}, \cdots,  e_m^{\ast}\}$,  which is a self-dual lattice of  $(\mathscr{W}, 2\langle, \rangle)$ with respect to $\Psi$.

  Let $H_{B'}(\mathcal{L}')=\mathcal{L}'\times K$ denote   the corresponding Heisenberg subgroup of $H_{B'}(\mathscr{W})$. Let $\Psi_{\mathcal{L}'}$ be the extended  character of $H_{B'}(\mathcal{L}')$ defined as: $(l, t) \longmapsto \Psi(t)$, for $l\in \mathcal{L}'$, $t \in K$. Then  $(\Pi'_{\Psi}=\cInd_{H_{B'}(\mathcal{L}')}^{H_{B'}(\mathscr{W})} \Psi_{\mathcal{L}'},  S_{\mathcal{L}^{'}}=\cInd_{H_{B'}(\mathcal{L}')}^{H_{B'}(\mathscr{W})} \C)$ defines a Heisenberg representation of $H_{B'}(\mathscr{W})$  in this case. By Weil's result, $\Pi'_{\Psi}$ can extend to $\overline{\Sp(\mathscr{W})}$ and $\mathscr{W}^{\vee}$.
\begin{lemma}
There exist the following isomorphisms:
\begin{itemize}
\item[(1)] $\alpha_{2}:    \Sp(\mathscr{W})\ltimes H_{B}(\mathscr{W}) \longrightarrow  \Sp(\mathscr{W}) \ltimes H_{B'}(\mathscr{W}),
 \quad  [g,(x,x^{\ast}; t)]  \longmapsto [g, (x, x^{\ast};2t)]$;
\item[(2)] $\alpha_{2}:    \Ps(\mathscr{W}) \longrightarrow \Ps'(\mathscr{W}),
 \quad  [g,q]  \longmapsto [g,2q]$.
\end{itemize}
\end{lemma}
\begin{proof}
Straightforward.
\end{proof}
\subsubsection{Special lattices} Let $$\mathcal{L}^{'\frac{1}{2}}=\Span_{\mathfrak{O}_K}\{ e'_1, \cdots, e'_m; 2 e_1^{\ast'}, \cdots,  2e_m^{\ast'}\}=\Span_{\mathfrak{O}_K}\{ e_1, \cdots, e_m;  e_1^{\ast}, \cdots,  e_m^{\ast}\}.$$   Let $V_{\psi}$ be the  subspace of $S_{\mathcal{L}'}$ of functions supported on $\mathcal{L}^{'\frac{1}{2}}$.  For any $f\in V_{\psi}$, $w\in \mathcal{L}^{'\frac{1}{2}}$, $a\in \mathcal{L}'$,
 $$f(a+w)=\Psi(-B'( a, w))f(w)=\Psi(-2B( a, w))f(w)=f(w).$$ So $f$ is indeed a function on $\mathcal{L}^{'\frac{1}{2}}/\mathcal{L}'$.  By observation, $\dim V_{\psi}= 2^{dm}$, $2^{d}=|F|$.

  Let us   treat $\mathcal{L}^{'\frac{1}{2}}$, $\mathcal{L}$  also as   two  free  modules over $\mathfrak{O}$ and $\mathfrak{m}$ respectively  under the basis $\{ e_1', \cdots, e'_m;  2e_1^{\ast'}, \cdots,  2e_m^{\ast'}\}$, endowed with the symplectic form $\langle, \rangle'$, and  written  by $\mathscr{W}_{ \mathfrak{O}}$ and $\mathscr{W}_{ \mathfrak{m}}$ respectively.
  Let $H_{B'}(\mathscr{W}_{ \mathfrak{O}})=\mathscr{W}_{ \mathfrak{O}} \times \mathfrak{O}$, $H_{B'}(\mathscr{W}_{\mathfrak{m}})=\mathscr{W}_{ \mathfrak{m}} \times \mathfrak{m}^{2}$ be the corresponding  subgroups  of $H_{B'}(\mathscr{W})$. Let  $R=  \mathfrak{O}/ \mathfrak{m}^2$.
 \begin{lemma}\label{HH}
 \begin{itemize}
 \item[(1)] $H_{B'}(\mathscr{W}_{ \mathfrak{m}})$ is a normal subgroup of $H_{B'}(\mathscr{W}_{ \mathfrak{O}})$.
 \item[(2)] The quotient  $H_{B'}(\mathscr{W}_{ \mathfrak{O}})/H_{B'}(\mathscr{W}_{ \mathfrak{m}})$ is isomorphic with $H_{\beta}(W)$, given in Section \ref{GLGU}.
 \item[(3)] $H_{B'}(\mathscr{W}_{ \mathfrak{m}})$ acts trivially on $V_{\psi}$.
 \item[(4)] $H_{B'}(\mathscr{W}_{ \mathfrak{O}})$ acts stably on $V_{\psi}$.
 \end{itemize}
 \end{lemma}
 \begin{proof}
 1) For $[w, t]\in H_{B'}(\mathscr{W}_{ \mathfrak{O}})$, $[w,t]^{-1}=[-w, -t+B'(w,w)]$. For $[w', t']\in H_{B'}(\mathscr{W}_{ \mathfrak{m}})$, $[w,t][w', t'][w,t]^{-1}=[w+w', t+t'+B'(w,w')][-w, -t+B'(w,w)]=[w', t'+B'(w,w')+B'(w,w)-B'(w+w',w)]=[w', t'+B'(w,w')-B'(w',w)]\in H_{B'}(\mathscr{W}_{ \mathfrak{m}})$.\\
 2) As sets, $H_{B'}(\mathscr{W}_{ \mathfrak{O}})/H_{B'}(\mathscr{W}_{ \mathfrak{m}})\simeq W \times R$.  For $[w, t], [w', t']\in H_{B'}(\mathscr{W}_{ \mathfrak{O}})$, let $[\overline{w}, \overline{t}],  [\overline{w}', \overline{t}']$ denote their images in $W \times R$. Then $[\overline{w}, \overline{t}] [\overline{w}', \overline{t'}]=[\overline{w+w'}, \overline{t}+\overline{t'}+ \overline{2B(w,w')}]=[\overline{w}+\overline{w'}, \overline{t}+\overline{t'}+ \beta(\overline{w},\overline{w'})]$.\\
 3) For $f\in V_{\psi}$, $h=[w',t']\in H_{B'}(\mathscr{W}_{ \mathfrak{m}})$, then $w'\in \mathcal{L}'$, $t'\in \mathfrak{m}^2$. Hence $\Pi'_{\Psi}(h) f=\Pi'_{\Psi}(t') \Pi'_{\Psi}(w') f=f$. \\
 4) If $f\in V_{\psi}$, $\supp f\subseteq \mathcal{L}^{'\frac{1}{2}}$. For $h=[w',t']\in H_{B'}(\mathscr{W}_{ \mathfrak{O}})$, $[\Pi'_{\psi}(h) f](w)=\Psi(t'+B'(w,w')) f(w+w')$, so $\supp [\Pi'_{\Psi}(h) f] \subseteq \mathcal{L}^{'\frac{1}{2}}$.
 \end{proof}
The restriction of $\Psi$ to $\mathfrak{O}$ defines a character of $R$, denoted by $\psi$. Then we also choose $\Psi$ such that  $\psi$ is a \emph{faithful} representation.
\begin{corollary}
The action of $H_{B'}(\mathscr{W}_{ \mathfrak{O}})/H_{B'}(\mathscr{W}_{ \mathfrak{m}})$ on  $V_{\psi}$ realizes the   Heisenberg representation of $H_{\beta}(W)$ attached to the  central character $\psi$.
\end{corollary}
\subsection{Extending to an affine symplectic group}\label{ASG}
Let $P's(\mathscr{W}_{ \mathfrak{O}})$ be the set  of elements $(g, q')$ of $P's(\mathscr{W})$ such that $g\in \Sp(\mathscr{W}_{ \mathfrak{O}})$, and $q':
\mathscr{W}_{ \mathfrak{O}} \longrightarrow \mathfrak{O}$, $q':
\mathscr{W}_{ \mathfrak{m}} \longrightarrow \mathfrak{m}^2$, $q':
\frac{\mathscr{W}_{ \mathfrak{O}}}{\mathscr{W}_{ \mathfrak{m}}} \longrightarrow \frac{ \mathfrak{O}}{\mathfrak{m}^2}$ is a quadratic function.
Let $\Gamma_1'=\ker( \Sp(\mathscr{W}_{ \mathfrak{O}}) \longrightarrow \Sp(\mathscr{W}_{ \mathfrak{O}}/\mathscr{W}_{ \mathfrak{m}}))$, and $P's(\mathscr{W}_{ \mathfrak{O}})_1=\{ (g, q') \in P's(\mathscr{W}_{ \mathfrak{O}}) \mid g \in \Gamma_1', q': \mathscr{W}_{ \mathfrak{O}} \longrightarrow \mathfrak{m}^2\}$. Recall the group homomorphism $\alpha': \Sp(\mathscr{W}_{ \mathfrak{O}}) \longrightarrow  P's(\mathscr{W}); g\longmapsto (g, q_g')$, where $q_g'(x+x^{\ast})= \frac{1}{2}B'(
xa, xb)+\frac{1}{2}B'( x^{\ast}c, x^{\ast}d)+B'(x^{\ast}c, xb)$, for $w\in\mathscr{W}_{ \mathfrak{O}}$,  $g=\begin{pmatrix} a& b\\c & d\end{pmatrix}\in \Sp(\mathscr{W}_{ \mathfrak{O}} )$. Hence the image lies in $P's(\mathscr{W}_{ \mathfrak{O}})$.

Let  $\mathscr{W}_{ \mathfrak{O}}^{'\vee}=\{ q: \mathscr{W}_{ \mathfrak{O}} \longrightarrow \mathfrak{O},  \textrm{ a  continuous  additve group homomorphism   }; q(\mathscr{W}_{ \mathfrak{m}}) \subseteq \mathfrak{m}^2; q: \frac{\mathscr{W}_{ \mathfrak{O}} }{\mathscr{W}_{ \mathfrak{m}}}\longrightarrow \frac{\mathfrak{O}}{\mathfrak{m}^2}, \textrm{ a quadratic function.}\}$,   $\mathscr{W}_{ \mathfrak{O}}^{''\vee}=\Hom_{group}( \mathscr{W}_{ \mathfrak{O}}, \mathfrak{m}^2)$, and $W^{\vee}=\{ f\mid f: W\longrightarrow R, \textrm{ a quadratic function}\}$.
\begin{lemma}\label{splitD}
  There exists a split exact sequence: $1\longrightarrow \mathscr{W}_{ \mathfrak{O}}^{'\vee} \longrightarrow P's(\mathscr{W}_{ \mathfrak{O}}) \longrightarrow \Sp(\mathscr{W}_{ \mathfrak{O}}) \longrightarrow 1$.
\end{lemma}
\begin{proof}
It follows from Lemma \ref{split} and the fact that the image of the section map lies in $P's(\mathscr{W}_{ \mathfrak{O}})$.
\end{proof}
\begin{lemma}\label{sec}
For any $g\in \Gamma_1'$, there exists $q'$  such that $(g, q') \in P's(\mathscr{W}_{ \mathfrak{O}})_1$.
\end{lemma}
\begin{proof}
1) For $g\in \Gamma_1'$, let us write $g=1+2g_1$, for some $g_1\in\End( \mathscr{W}_{ \mathfrak{O}})$. Recall $\alpha(g)=(g,q_g')$.  So for any $w\in \mathscr{W}_{\mathfrak{O}}$,
$$q'_g(w)=\frac{1}{2}B'(wg,wg)-\frac{1}{2}B'(w,w)=\frac{1}{2}B'(2wg_1,2wg_1)+\frac{1}{2}B'(2wg_1,w)+\frac{1}{2}B'(w, 2wg_1) $$
$$=2B'(wg_1,wg_1)+\langle wg_1,wg\rangle'+B'(w, 2wg_1).$$
Note that $2B'(wg_1,wg_1)+B'(w, 2wg_1) \in \mathfrak{m}^2$. On the other hand, let us define $f(w)=\langle wg_1,wg\rangle'=\langle 2wg_1,wg\rangle =-\langle w,wg\rangle\in \mathfrak{m}$. Then $f|_{\mathscr{W}_{ \mathfrak{m}}}\in \mathfrak{m}^2$, and $f(w_1+w_2)=f(w_1)+f(w_2) \textrm{ modulo } \mathfrak{m}^2$.\\
2) Recall that $\mathscr{W}_{\mathfrak{O}}=\mathfrak{O} e_1'+ \cdots + \mathfrak{O} e_m' + \mathfrak{O} 2e_1^{\ast'}+ \cdots + \mathfrak{O} 2e_m^{\ast'}$. Let us define an additive group homomorphism:
$$f': \mathscr{W}_{\mathfrak{O}}  \longrightarrow \mathfrak{O}; \sum a_i  e_i' + a_j^{\ast} 2e_j^{\ast'} \longmapsto \sum a_i f(e_i')+ a_j^{\ast} f(2e_j^{\ast'}).$$
Then: \begin{itemize}
\item[(a)] $f'(\mathscr{W}_{\mathfrak{O}}) \subseteq  \mathfrak{m}$,  $f'(\mathscr{W}_{\mathfrak{m}}) \subseteq  \mathfrak{m}^2$, so
$f'|_{\mathscr{W}_{ \mathfrak{O}}} \equiv f'|_{\mathscr{W}_{ \mathfrak{O}}/\mathscr{W}_{ \mathfrak{m}}}(\mod  \mathfrak{m}^2)$;
\item[(b)] for $w=\sum a_i  e_i' + a_j^{\ast} 2e_j^{\ast'}$, $$f'(w)\equiv \sum a_i f(e_i')+ a_j^{\ast} f(2e_j^{\ast'}) \equiv \sum a_i^2 f(e_i')+ a_j^{\ast 2} f(2e_j^{\ast'}) \equiv \sum  f(a_ie_i')+ f( a_j^{\ast}2e_j^{\ast'}) \equiv f(w)(\mod  \mathfrak{m}^2).$$
\end{itemize}
Hence $(g, q_g'+f')\in P's(\mathscr{W}_{ \mathfrak{O}}) $. Moreover, for $w\in \mathscr{W}_{ \mathfrak{O}}$,
$$ q_g'(w)+f'(w)\equiv 2B'(wg_1,wg_1)+2\langle wg_1,wg\rangle'+B'(w, 2wg_1) \equiv 0 (\mod  \mathfrak{m}^2).$$
\end{proof}
\begin{corollary}\label{exgamma1}
  There exists an exact sequence: $1\longrightarrow \mathscr{W}_{ \mathfrak{O}}^{''\vee} \longrightarrow P's(\mathscr{W}_{ \mathfrak{O}})_1 \longrightarrow \Gamma_1' \longrightarrow 1$.
\end{corollary}

\begin{lemma}\label{mainlmm1}
\begin{itemize}
\item[(1)] $P's(\mathscr{W}_{ \mathfrak{O}})_1$ is a normal subgroup of $P's(\mathscr{W}_{ \mathfrak{O}})$.
\item[(2)] $P's(\mathscr{W}_{ \mathfrak{O}})/P's(\mathscr{W}_{ \mathfrak{O}})_1$ is isomorphic with $ASp(W)$, which is  given in Section \ref{GLGU}.
\end{itemize}
\end{lemma}
\begin{proof}
1) For $h_1=(g_1, q'_1) \in P's(\mathscr{W}_{ \mathfrak{O}})_1$, $h=(g, q') \in P's(\mathscr{W}_{ \mathfrak{O}})$, $h^{-1}=(g^{-1}, -{}^{g^{-1}}q')$, where $-{}^{g^{-1}}q'(w)=-q'(wg^{-1})$. Assume $hh_1h^{-1}=(gg_1g^{-1}, q'')$. Then $q''=q'+{}^gq'_{1}-{}^{gg_1g^{-1}}q'$, and $q''(w)=q'(w)+q'_1(wg)-q'(wgg_1g^{-1})$. Note that the image of  ${}^gq'_1$ lies in $\mathfrak{m}^2$.
Assume $g_1=1+2g_2$, for some $g_2\in  \End( \mathscr{W}_{ \mathfrak{O}})$.   For $w\in \mathscr{W}_{ \mathfrak{O}}$, let $w_1=wg$. Then:
$$q'(w)-q'(wgg_1g^{-1})=q'(w_1g^{-1})-q'(w_1g_1g^{-1})$$
$$=-q'(2w_1g_2g^{-1})-B'(2w_1g_2,w_1)+B'(2w_1g_2g^{-1},w_1g^{-1})\in \mathfrak{m}^2.$$
2) Follow the notations from Section \ref{GLGU}.  Let us choose the basis  $\{ e_1, \cdots, e_m;  e_1^{\ast}, \cdots,  e_m^{\ast}\}$ of $\mathscr{W}_{ \mathfrak{O}}$ as well as $\mathscr{W}_{ \mathfrak{m}}$. Then identity  $\mathscr{W}_{ \mathfrak{O}}/\mathscr{W}_{ \mathfrak{m}}$ with $W$(Section \ref{GLGU}). Then there exists the following commutative diagram:
\begin{equation}\label{qeq}
\begin{CD}
@. 1@. 1 @.1 @. \\
@.@VVV @VVV @VVV @.\\
1 @>>> \mathscr{W}_{ \mathfrak{O}}^{''\vee} @>>> P's(\mathscr{W}_{ \mathfrak{O}})_1@>>> \Gamma_1' @>>> 1\\
@.@VVV @VVV @VVV @.\\
1 @>>> \mathscr{W}_{ \mathfrak{O}}^{'\vee} @>>> P's(\mathscr{W}_{ \mathfrak{O}})@>>> \Sp(\mathscr{W}_{ \mathfrak{O}}) @>>> 1\\
@.@VVV @VVV @VVV @.\\
1@>>>W^{\vee} @>>> ASp(W) @>>>\Sp(W) @>>> 1\\
@.@VVV @VVV @VVV @.\\
@. 1@. 1 @.1 @.
\end{CD}
\end{equation}
\end{proof}

\begin{lemma}\label{mainlmm2}
$(\Pi'_{\Psi}, V_{\psi})$ can extend to be a projective representation of $P's(\mathscr{W}_{ \mathfrak{O}})\ltimes H_{B'}(\mathscr{W}_{ \mathfrak{O}})$.
\end{lemma}
\begin{proof}
By Lemmas \ref{cc2}, \ref{HH}, it suffices to check that the actions of $ \mathscr{W}_{ \mathfrak{O}}^{'\vee}$ and $\Sp(\mathscr{W}_{ \mathfrak{O}})$ are comparable. More precisely, for $(q,g)\in \mathscr{W}_{ \mathfrak{O}}^{'\vee} \rtimes \Sp(\mathscr{W}_{ \mathfrak{O}})\simeq P's(\mathscr{W}_{ \mathfrak{O}})$, we have $(g,q)=(1,q)(g,0)=(g,0)(1,{}^{g^{-1}}q)$.  It reduces to checking that
\begin{equation}\label{Psi}
\Pi'_{\Psi}(q)\Pi'_{\Psi}(g)f=\Pi'_{\Psi}(g)\Pi'_{\Psi}({}^{g^{-1}}q)f, \qquad \quad f\in V_{\psi}.
\end{equation}

 By the formulas in the above case 4, for $q\in \mathscr{W}_{ \mathfrak{O}}^{'\vee}$, $f\in V_{\psi}$, $\Pi'_{\Psi}(q)f(w)=\Psi(q(w))f(w)$. By the formulas in the above  cases 1, 2, for $g=h_a, u(b) \in  \Sp(\mathscr{W}_{ \mathfrak{O}})$, $\Pi'_{\Psi}(g)f(w)=\Psi(q_g(w))f(wg)$.  For $g=h_{(S,1/2)}\omega_S \in \Sp(\mathscr{W}_{ \mathfrak{O}})$, by the formulas in the above case 3, we know $\Pi'_{\Psi}(\omega_S)f(w)=\Psi(q_{\omega_S}(w))f(w\omega_S)$.  Hence, for such elements, the above equation (\ref{Psi}) is satisfied.  For $g=h_{(S,1/2)}=\begin{pmatrix}
  a&\\
   & (a^{\ast})^{-1}
\end{pmatrix}\in \Sp(\mathscr{W})$, for simplicity, let us go back to the formulas (\ref{representationsp3}), (\ref{representationsp4}). Then for $f\in S(\mathcal{X})$, and $\supp f\subseteq \tfrac{1}{2}\mathcal{X}\cap \mathcal{L}'$,
 $$\Pi'_{\Psi}(q)\Pi'_{\Psi}(g)f(x)=q(x)[\Pi'_{\Psi}(g)f](x)=q(x)|det_{\mathcal{X}}(a)|^{\tfrac{1}{2}} f(xa),$$
  $$\Pi'_{\Psi}(g)\Pi'_{\Psi}({}^{g^{-1}}q)f(x)=|det_{\mathcal{X}}(a)|^{\tfrac{1}{2}}[\Pi'_{\Psi}({}^{g^{-1}}q)f](xa)=q(x)|det_{\mathcal{X}}(a)|^{\tfrac{1}{2}} f(xa).$$
\end{proof}
Let us see the behavior of the action of
 $P's(\mathscr{W}_{ \mathfrak{O}})_1 \ltimes H_{B'}(\mathscr{W}_{ \mathfrak{m}})$ on $V_{\psi}$. By Lemma \ref{HH}, $H_{B'}(\mathscr{W}_{ \mathfrak{m}})$ acts trivially on $V_{\psi}$.   Recall the group homomorphism:
$$\alpha': \Sp(\mathscr{W}_{ \mathfrak{O}}) \longrightarrow  P's(\mathscr{W}_{ \mathfrak{O}}); g\longmapsto (g, q_g').$$  By Coro.\ref{exgamma1}, there exists an exact sequence:   $$1\longrightarrow \mathscr{W}_{ \mathfrak{O}}^{''\vee} \longrightarrow P's(\mathscr{W}_{ \mathfrak{O}})_1 \longrightarrow \Gamma_1' \longrightarrow 1.$$
Then there exists a section map: $$\alpha'': \Gamma_1' \longrightarrow P's(\mathscr{W}_{ \mathfrak{O}})_1; g \longmapsto (g, q_g'').$$
By Lemma \ref{sec},  we can choose $q_g''(w)\equiv 2q_g'(w) (\mod \mathfrak{m}^2)$, for $w\in \mathscr{W}_{ \mathfrak{O}}$. By the formulas in the above case 4, $\mathscr{W}_{ \mathfrak{O}}^{''\vee}$ acts trivially on $V_{\psi}$. It reduces to discussing the group $\Gamma_1'$.   Under the basis $\{e_1', \cdots, e_m'; e_1^{'\ast}, \cdots, e_m^{'\ast}\}$, $\Sp(  \mathcal{L}^{'\tfrac{1}{2}})\simeq \begin{pmatrix}
  1&0\\
  0 & \tfrac{1}{2}
\end{pmatrix}\Sp_{2m}(\mathfrak{O}) \begin{pmatrix}
  1&0\\
  0 &2
\end{pmatrix} $, and $\Gamma_1' \simeq \begin{pmatrix}
  1&0\\
  0 & \tfrac{1}{2}
\end{pmatrix}\Gamma_1 \begin{pmatrix}
  1&0\\
  0 &2
\end{pmatrix}$, where $\Gamma_1$ is given in Corollary \ref{Gamma}.

Let $h_{1/2}=\begin{pmatrix}
  1&0\\
  0 & \tfrac{1}{2}
\end{pmatrix}$. By the  Iwahori decomposition, $\Gamma_1'$ can be generated by the elements of the forms $h_{a}=\begin{pmatrix}
  a&0\\
  0 & (a^{\ast})^{-1}
\end{pmatrix}$, $u(b)=\begin{pmatrix}
  1&b\\
  0 & 1
\end{pmatrix}$, and $u^{-}(c)=\begin{pmatrix}
  1&0\\
  c & 1
\end{pmatrix}$, for  $a\in I+  M_m(\mathfrak{m})$, certain $b\in  M_m(\mathfrak{m}^2) $, and certain $c\in M_m(\mathfrak{O})$.  Let $f\in V_{\psi}$, $w=x+x^{\ast}\in \mathcal{L}^{'\tfrac{1}{2}}$. On $\Gamma_1'$, we choose the section $\alpha'': g \longmapsto (g,q_g'')$.

\begin{itemize}
\item[(1)] For $g=h_a$, by Case 2, $\Pi'_{\Psi}(g) f(w)= f(wg)$.
\item[(2)] For $g=u(b)$, by Case 1, $\Pi'_{\Psi}(g) f(w)=\Psi(q''_g(w)) f(wg)=f(x+x^{\ast}+xb)=f(w)$.
\item[(3)] For $g=u^{-}(c)$, let us write  $h=\omega^{-1} u^{-}(c)\omega$, and $g=\omega h\omega^{-1}$.  On $h$, we choose the section $\alpha''(h)=(h, q''_h)$, for $q''_h(w)=2q'_h(w) \mod \mathfrak{m}^2$, for $w\in \mathscr{W}_{ \mathfrak{O}}$. Then by Case 1, $\Pi_{\Psi}'(\alpha''(h))=\Psi(2q_h'(w)) f(wh)$. Similarly, on $\omega$, $\omega^{-1}$, we also choose the section $\alpha''$. Then:
    $$\Pi_{\Psi}'(\alpha''(g))c_{PR, \mathcal{X}^{\ast}}(\omega,h\omega^{-1})f(w)=\Pi_{\Psi}'(\alpha''(\omega))\Pi_{\Psi}'(\alpha''(h))\Pi_{\Psi}'(\alpha''(\omega^{-1}))f(w)$$
    $$=\Psi(2q_{g}'(w)f(wg)=f(w).$$
Let  $\mu_g=c_{PR, \mathcal{X}^{\ast}}(\omega,h\omega^{-1})$. Then: $\Pi'_{\Psi}(g)\mu_gf(w)=f(w)$.
\item[(4)] For any $g\in \Gamma_1'$, let us write $g=g_1g_2$, for some unique $g_1\in h_{1/2}(N^-\cap \Gamma_1) h_{1/2}^{-1}$,  $g_2\in h_{1/2}(B\cap \Gamma_1)h_{1/2}^{-1}$. Let us define a function $t': \Gamma_1' \longrightarrow \mu_8; g \longmapsto \mu_{g_1}$.  Let us modify the action of $\Gamma_1'$  by defining
$\Pi^{'\sim}_{\Psi}(g)f=\Pi'_{\Psi}(g) t'(g)f$, for $f\in V_{\psi}$.  Then $P's(\mathscr{W}_{ \mathfrak{O}})_1 \ltimes H_{B'}(\mathscr{W}_{ \mathfrak{m}})$ acts trivially on $V_{\psi}$.
\end{itemize}

Recall the exact sequence:
$$1\longrightarrow \Gamma_1' \longrightarrow \Sp(\mathscr{W}_{\mathfrak{O}}) \longrightarrow \Sp(\mathscr{W}_{\mathfrak{O}}/\mathscr{W}_{ \mathfrak{m}}) \longrightarrow 1.$$  Under the basis $\{e_1', \cdots, e_m'; e_1^{'\ast}, \cdots, e_m^{'\ast}\}$, $\Sp(  \mathscr{W}_{\mathfrak{O}})\simeq
\begin{pmatrix}
  1&0\\
  0 & \tfrac{1}{2}
\end{pmatrix}\Sp_{2m}(\mathfrak{O}) \begin{pmatrix}
  1&0\\
  0 &2
\end{pmatrix}=h_{1/2}\Sp_{2m}(\mathfrak{O})h_{1/2}^{-1} $.  Recall the Iwahori decomposition in Section \ref{Iwode},  $\Sp_{2m}(\mathfrak{O})=\cup_{\tau \in \mathfrak{W}} I\tau I$, and $I=(N^-\cap I)(B\cap I)$, and $\Gamma_1=(N^-\cap \Gamma_1)(B\cap \Gamma_1)$. Recall $W=\mathscr{W}_{\mathfrak{O}}/\mathscr{W}_{ \mathfrak{m}}$, and the Bruhat decomposition $\Sp(W)=\cup_{\tau \in \mathfrak{W}} B(W) \tau N(W)$.

 For $g\in  \Sp(  \mathscr{W}_{\mathfrak{O}})$, let $\dot{g}$ be its image in $\Sp(W)$.  Let us choose a section map: $s: \Sp(W) \longrightarrow \Sp(  \mathscr{W}_{\mathfrak{O}})$ such that it is comparable with the Iwahori decomposition. More precisely, $s(B(W))\subseteq
h_{1/2}(B\cap I) h_{1/2}^{-1}$,  $s(T(W))\subseteq
h_{1/2}(T\cap I) h_{1/2}^{-1}$, $s(N(W))\subseteq
h_{1/2}(N\cap I) h_{1/2}^{-1}$,  $s(\tau)=h_{1/2}\tau h_{1/2}^{-1}$, and $s(b\tau n)=s(b)s(\tau)s(n)$. Note that $s(\dot{g}_1)=1$, for $g_1\in \Gamma_1'$. Let $W=X\oplus X^{\ast}$ be the corresponding  complete polarization. Then $P(X^{\ast})=\cup_{\tau\in S_m} B(W) \tau N(W)$, and $s(P(X^{\ast})) \subseteq P(\mathcal{X}^{\ast})$.

Assume  $g=s(\dot{g})g_1$, for some $g_1\in \Gamma_1'$.  Let us define a function as follows:
$$t': \Sp(  \mathscr{W}_{\mathfrak{O}})\longrightarrow \mu_8; g \longmapsto \left\{ \begin{array}{ccl} c_{PR, \mathcal{X}^{\ast}}(s(\dot{g}),g_1)t'(g_1) & &g=s(\dot{g})g_1 \notin \Gamma_1',\\
t'(g_1) & & g\in \Gamma_1'. \end{array} \right.$$  Let us modify the action of $ \Sp(\mathscr{W}_{\mathfrak{O}})$  on $V_{\psi}$ by defining
$$\Pi^{'\sim}_{\Psi}(g)f=\Pi'_{\Psi}(g) t'(g)f, \qquad f\in V_{\psi}.$$
For $g=s(\dot{g})g_1\in  \Sp(\mathscr{W}_{\mathfrak{O}})$, if $s(\dot{g})=1$, then by the above (4), $\Pi^{'\sim}_{\Psi}(g)f=f$; if $s(\dot{g})\neq 1$, then $$\Pi^{'\sim}_{\Psi}(g)f=\Pi'_{\Psi}(g) t'(g)f= c_{PR, \mathcal{X}^{\ast}}(s(\dot{g}),g_1)^{-1} \Pi'_{\Psi}(s(\dot{g}))\Pi'_{\Psi}(g_1)t'(g)f=\Pi'_{\Psi}(s(\dot{g}))\Pi^{'\sim}_{\Psi}(g_1)f=\Pi'_{\Psi}(s(\dot{g}))f.$$
Therefore, for  $g=s(\dot{g})g_1\in \Sp(  \mathscr{W}_{\mathfrak{O}}), g_2\in \Gamma_1'$, we have $gg_2=s(\dot{g})g_1g_2$, and $g_2g=s(\dot{g})[(s(\dot{g})^{-1}g_1s(\dot{g}))g_2]$. Moreover,
$$\Pi^{'\sim}_{\Psi}(gg_2)f=\Pi'_{\Psi}(s(\dot{g}))f=\Pi^{'\sim}_{\Psi}(g)f=\Pi^{'\sim}_{\Psi}(g)\Pi^{'\sim}_{\Psi}(g_2)f,$$
$$\Pi^{'\sim}_{\Psi}(g_2g)f=\Pi'_{\Psi}(s(\dot{g}))f=\Pi^{'\sim}_{\Psi}(g)f=\Pi^{'\sim}_{\Psi}(g_2)\Pi^{'\sim}_{\Psi}(g)f.$$
Let $c_{\Pi^{'\sim}_{\Psi}}(-,-)$ be  the $2$-cocycle attached to $\Pi^{'\sim}_{\Psi}$, i.e. $$\Pi^{'\sim}_{\Psi}(g_1)\Pi^{'\sim}_{\Psi}(g_2)=c_{\Pi^{'\sim}_{\Psi}}(g_1,g_2)\Pi^{'\sim}_{\Psi}(g_1g_2), \quad g_1, g_2\in \Sp(  \mathscr{W}_{\mathfrak{O}}).$$  In particular, for $g_2\in \Gamma_1'$, $c_{\Pi^{'\sim}_{\Psi}}(g_1,g_2)=1=c_{\Pi^{'\sim}_{\Psi}}(g_2,g_1)$.
\begin{lemma}\label{GG}
\begin{itemize}
\item[(1)] For $g=g_2s(\dot{g})$, $t'(g)=t'(g_2)c_{PR, \mathcal{X}^{\ast}}(g_2,s(\dot{g}))$, for $g_2\in \Gamma_1'$, $g\in \Sp(  \mathscr{W}_{\mathfrak{O}})$.
\item[(2)] For $g=g_2s(\dot{g})g_1$, $t'(g)=t'(g_2)t'(g_1)c_{PR, \mathcal{X}^{\ast}}(g_2,s(\dot{g}))c_{PR, \mathcal{X}^{\ast}}(g_2s(\dot{g}), g_1)$, for $g_1, g_2\in \Gamma_1'$, $g\in \Sp(  \mathscr{W}_{\mathfrak{O}})$.
\item[(3)] $t'(p_1h_{1/2}\tau h_{1/2}^{-1}p_2)=1$, for $ p_1,p_2\in   h_{1/2}(B\cap I) h_{1/2}^{-1}$, $\tau\in \mathfrak{W}$.
\end{itemize}
\end{lemma}
\begin{proof}
1) $\Pi'_{\Psi}(g) t'(g)=\Pi^{'\sim}_{\Psi}(g)= \Pi^{'\sim}_{\Psi}(g_2)\Pi^{'\sim}_{\Psi}(s(\dot{g}))=\Pi^{'}_{\Psi}(g_2)t'(g_2) \Pi^{'}_{\Psi}(s(\dot{g}))=t'(g_2)c_{PR, \mathcal{X}^{\ast}}(g_2,s(\dot{g}))\Pi^{'}_{\Psi}(g)$. Hence
$t'(g)=t'(g_2)c_{PR, \mathcal{X}^{\ast}}(g_2,s(\dot{g}))$.\\
2) $\Pi'_{\Psi}(g) t'(g)=\Pi^{'\sim}_{\Psi}(g)= \Pi^{'\sim}_{\Psi}(g_2)\Pi^{'\sim}_{\Psi}(s(\dot{g}))\Pi^{'\sim}_{\Psi}(g_1)=\Pi^{'}_{\Psi}(g_2)t'(g_2) \Pi^{'}_{\Psi}(s(\dot{g}))\Pi^{'}_{\Psi}(g_1)t'(g_1) =t'(g_2)t'(g_1)c_{PR, \mathcal{X}^{\ast}}(g_2,s(\dot{g}))c_{PR, \mathcal{X}^{\ast}}(g_2s(\dot{g}), g_1))\Pi^{'}_{\Psi}(g)$. Hence
$$t'(g)=t'(g_2)t'(g_1)c_{PR, \mathcal{X}^{\ast}}(g_2,s(\dot{g}))c_{PR, \mathcal{X}^{\ast}}(g_2s(\dot{g}), g_1).$$
3)  Assume $p_i=s(\dot{p}_i)p_i'$. Then, $p_i'\in h_{1/2}(B\cap \Gamma_1) h_{1/2}^{-1}$. So  $t'(p_i)=c_{PR, \mathcal{X}^{\ast}}(s(\dot{p}_i),p_i')=1$. Let $q= p_1h_{1/2}\tau h_{1/2}^{-1}p_2$. Then $q=p_1'' s(q) p_2''$, for some $p_i''\in h_{1/2}(B\cap \Gamma_1) h_{1/2}^{-1}$. By (2), we have:
$$t'(q)=t'(p_1'')t'(p_2'')c_{PR, \mathcal{X}^{\ast}}(p_1'',s(\dot{q}))c_{PR, \mathcal{X}^{\ast}}(p_1''s(\dot{q}), p_2'')=1.$$
\end{proof}
 Moreover, for $g_1, g_2\in \Sp(  \mathscr{W}_{\mathfrak{O}})$,
$$\Pi^{'\sim}_{\Psi}(g_1)\Pi^{'\sim}_{\Psi}(g_2)f=\Pi'_{\Psi}(s(\dot{g}_1))\Pi'_{\Psi}(s(\dot{g}_2))f=c_{PR, \mathcal{X}^{\ast}}(s(\dot{g}_1),s(\dot{g}_2))\Pi'_{\Psi}(s(\dot{g}_1)s(\dot{g}_2))f$$
$$= c_{PR, \mathcal{X}^{\ast}}(s(\dot{g}_1),s(\dot{g}_2)) \Pi'_{\Psi}(s(\dot{g}_1)s(\dot{g}_2))\Pi'_{\Psi}(s(\dot{g}_1\dot{g}_2))^{-1}  \Pi^{'\sim}_{\Psi}(g_1g_2)f.$$

Hence $c_{\Pi^{'\sim}_{\Psi}}(g_1,g_2)=c_{PR, \mathcal{X}^{\ast}}(s(\dot{g}_1),s(\dot{g}_2))\Pi'_{\Psi}(s(\dot{g}_1)s(\dot{g}_2))\Pi'_{\Psi}(s(\dot{g}_1\dot{g}_2))^{-1}  $, for $ g_1, g_2\in \Sp(  \mathscr{W}_{\mathfrak{O}})$. Note that $\Pi^{'\sim}_{\Psi}(s(\dot{g}_1)s(\dot{g}_2))=\Pi^{'\sim}_{\Psi}(s(\dot{g}_1\dot{g}_2))$. Hence
$$\Pi'_{\Psi}(s(\dot{g}_1)s(\dot{g}_2))=\Pi^{'\sim}_{\Psi}(s(\dot{g}_1)s(\dot{g}_2)) t'(s(\dot{g}_1)s(\dot{g}_2))^{-1}$$
$$=\Pi^{'\sim}_{\Psi}(s(\dot{g}_1\dot{g}_2))t'(s(\dot{g}_1)s(\dot{g}_2))^{-1}=\Pi^{'}_{\Psi}(s(\dot{g}_1\dot{g}_2))t'(s(\dot{g}_1)s(\dot{g}_2))^{-1}.$$
Therefore, $c_{\Pi^{'\sim}_{\Psi}}(g_1,g_2)=c_{PR, \mathcal{X}^{\ast}}(s(\dot{g}_1),s(\dot{g}_2))t'(s(\dot{g}_1)s(\dot{g}_2))^{-1}$.
\begin{lemma}\label{reduction}
Let $g_1, g_2 \in  \Sp(  \mathscr{W}_{\mathfrak{O}})$, $p, p_1, p_2\in   h_{1/2}(B\cap I) h_{1/2}^{-1}$.
\begin{itemize}
\item[(1)] $c_{\Pi^{'\sim}_{\Psi}}(p,g_1)=1=c_{\Pi^{'\sim}_{\Psi}}(g_1,p)$.
\item[(2)] $c_{\Pi^{'\sim}_{\Psi}}(p_1g_1p^{-1},pg_2p_2)=c_{\Pi^{'\sim}_{\Psi}}(g_1,g_2)$.
\item[(3)] $c_{\Pi^{'\sim}_{\Psi}}( h_{1/2}\omega_{S_1} h_{1/2}^{-1}, h_{1/2}\omega_{S_2} h_{1/2}^{-1})=1$.
\item[(4)] $t'(p_1g_1p_2)=t'(g_1)$, for $p_i\in \Gamma_1'\cap B$.
\end{itemize}
\end{lemma}
\begin{proof}
1) (a) $c_{\Pi^{'\sim}_{\Psi}}(p,g_1)=c_{PR, \mathcal{X}^{\ast}}(s(\dot{p}),s(\dot{g}_1))t'(s(\dot{p})s(\dot{g}_1))^{-1}$. Note that $ s(\dot{p})\in  B$, and $s(\dot{p})s(\dot{g}_1)=p_1s(\dot{p}\dot{g}_1)$, for some $p_1\in B\cap \Gamma_1'$. Hence $c_{PR, \mathcal{X}^{\ast}}(s(\dot{p}),s(\dot{g}_1))=1$. By Lemma \ref{GG}, $t'(s(\dot{p})s(\dot{g}_1))=t'(p_1s(\dot{p}\dot{g}_1))=t'(p_1)c_{PR, \mathcal{X}^{\ast}}(p_1,s(\dot{p}\dot{g}_1))=1$. \\
(b) $c_{\Pi^{'\sim}_{\Psi}}(g_1,p)=c_{PR, \mathcal{X}^{\ast}}(s(\dot{g}_1), s(\dot{p}))t'(s(\dot{g}_1)s(\dot{p}))^{-1}=t'(s(\dot{g}_1)s(\dot{p}))^{-1}$. Note that $s(\dot{g}_1)s(\dot{p})=s(\dot{g}_1\dot{p}) p_2$, for some $p_2\in B\cap \Gamma_1'$. So $t'(s(\dot{g}_1)s(\dot{p}))=t'(s(\dot{g}_1\dot{p}) p_2)=c_{PR, \mathcal{X}^{\ast}}(s(\dot{g}_1\dot{p}),p_2)t'(p_2)=1$. \\
2) It is a consequence of (1). More precisely,
$$c_{\Pi^{'\sim}_{\Psi}}(p_1g_1p^{-1},pg_2p_2)=c_{\Pi^{'\sim}_{\Psi}}(p_1,g_1p^{-1})c_{\Pi^{'\sim}_{\Psi}}(p_1g_1p^{-1},pg_2p_2)$$
$$=c_{\Pi^{'\sim}_{\Psi}}(p_1,g_1g_2p_2)c_{\Pi^{'\sim}_{\Psi}}(g_1p^{-1},pg_2p_2)=c_{\Pi^{'\sim}_{\Psi}}(g_1p^{-1},pg_2p_2)$$
$$=c_{\Pi^{'\sim}_{\Psi}}(g_1,p^{-1})c_{\Pi^{'\sim}_{\Psi}}(g_1p^{-1},pg_2p_2)=c_{\Pi^{'\sim}_{\Psi}}(g_1,g_2p_2)c_{\Pi^{'\sim}_{\Psi}}(p^{-1},pg_2p_2).$$
$$=c_{\Pi^{'\sim}_{\Psi}}(g_1,g_2p_2)=c_{\Pi^{'\sim}_{\Psi}}(g_1,g_2p_2)c_{\Pi^{'\sim}_{\Psi}}(g_2,p_2)$$
$$=c_{\Pi^{'\sim}_{\Psi}}(g_1,g_2)c_{\Pi^{'\sim}_{\Psi}}(g_1g_2, p_2)=c_{\Pi^{'\sim}_{\Psi}}(g_1,g_2).$$
3) Let us write $g_i= h_{1/2}\omega_{S_1} h_{1/2}^{-1}$. Note that $g_i=s(\dot{g_i})$, $g_1g_2=s(\dot{g_1}\dot{g_2})$. Hence $c_{\Pi^{'\sim}_{\Psi}}(g_1, g_2)=c_{PR, \mathcal{X}^{\ast}}(s(\dot{g}_1),s(\dot{g}_2))t'(s(\dot{g}_1)s(\dot{g}_2))^{-1}=t'(s(\dot{g}_1)s(\dot{g}_2))^{-1}=1$.\\
4) Note that $\Pi^{'\sim}_{\Psi}(p_i)=\Pi^{'}_{\Psi}(p_1)t'(p_1)=1$, so $\Pi^{'}_{\Psi}(p_1)=1$, $\Pi'_{\Psi}(p_1g_1p_2)=\Pi'_{\Psi}(g_1)$.  Then $\Pi^{'\sim}_{\Psi}(p_1g_1p_2)=\Pi^{'\sim}_{\Psi}(p_1)\Pi^{'\sim}_{\Psi}(g_1)\Pi^{'\sim}_{\Psi}(p_2)=\Pi^{'\sim}_{\Psi}(g_1)$. So $\Pi'_{\Psi}(p_1g_2p_2) t'(p_1g_1p_2)=\Pi'_{\Psi}(g_1)t'(g_1)$. Hence $t'(p_1g_1p_2)=t'(g_1)$.
\end{proof}
Note that  $P's(\mathscr{W}_{ \mathfrak{O}})\ltimes H_{B'}(\mathscr{W}_{ \mathfrak{O}})/ P's(\mathscr{W}_{ \mathfrak{O}})_1 \ltimes H_{B'}(\mathscr{W}_{ \mathfrak{m}}) \simeq  ASp(W) \ltimes H_{\beta}(W)$. Hence the action of $P's(\mathscr{W}_{ \mathfrak{O}})\ltimes H_{B'}(\mathscr{W}_{ \mathfrak{O}})$ on    $V_{\psi}$ via $\Pi^{'\sim}_{\Psi}$ realizes the projective representation of $ASp(W) \ltimes H_{\beta}(W)$ attached to $\psi$ as proved in \cite{GeLy}, \cite{GuHa}.
\subsection{$\mu_4$} Let us  verify that   the associated  $2$-cocycle  can be valued on $\mu_4$ as obtained in \cite{GeLy}, \cite{GuHa}.

For the character  $\Psi$ of $K$, let $\gamma$ denote the  corresponding Weil index. For $g_1, g_2 \in \Sp(\mathscr{W}_{ \mathfrak{O}})$, let $q(g_1,g_2)=q(\mathcal{X}^{\ast}, \mathcal{X}^{\ast} g_2^{-1}, \mathcal{X}^{\ast} g_1)$ be the corresponding Leray invariant. (cf. \cite[p.55]{MoViWa}, \cite{Ra},\cite{Pe}). Then $c_{PR, \mathcal{X}^{\ast}}(g_1, g_2)=\gamma(\Psi(\frac{q(g_1, g_2)}{2}))$. Recall $\Sp(\mathscr{W}_{ \mathfrak{O}}) \longrightarrow \Sp(\mathscr{W}_{ \mathfrak{O}}/\mathscr{W}_{ \mathfrak{m}}); g \longrightarrow \dot{g}$. For $g_1, g_2\in \Sp(\mathscr{W}_{ \mathfrak{O}})$, by the above discussion, $c_{\Pi^{'\sim}_{\Psi}}(g_1,g_2)=c_{PR, \mathcal{X}^{\ast}}(s(\dot{g}_1)),s(\dot{g}_2))t'(s(\dot{g}_1)s(\dot{g}_2))^{-1}$, which only depends on $\dot{g}_1, \dot{g}_2$.

According to \cite{Pe}, \cite{Ra}, by modifying $c_{\Pi^{'\sim}_{\Psi}}(-,-)$ with some  normalizing constants $m(g)$,  we can  get a cocycle that takes the values on $\mu_2$. More precisely,
for $g\in \Sp(\mathscr{W})$, and $g=p_1 \omega_S p_2$, we can  define $m(g)=\gamma(x(g), \tfrac{1}{2} \Psi)^{-1} \{\gamma(\tfrac{1}{2} \Psi)\}^{-|S|}$.  One can see \cite{Kud}, \cite{Ra} for the details.    Define the cocycle:
$$\widetilde{c_{PR, \mathcal{X}^{\ast}}}(g_1,g_2)=c_{PR, \mathcal{X}^{\ast}}(g_1,g_2) m(g_1)m(g_2)m(g_1g_2)^{-1}.$$
Then $\widetilde{c_{PR, \mathcal{X}^{\ast}}}(g_1,g_2)\in \mu_2$. Note that for $p\in P(\mathcal{X}^{\ast})$,  $m(p)\in \mu_4$.
 Associated to the cocycle $\widetilde{c_{PR, \mathcal{X}^{\ast}}}(-,-)$, we denote  the representation by $\widetilde{\Pi_{\Psi}'}$, i.e. $ \widetilde{\Pi_{\Psi}'}(g)=\Pi_{\Psi}'(g) m(g)$.

On $\Gamma_1'$, $\Pi^{'\sim}_{\Psi}(g)=\Pi_{\Psi}'(g)t'(g)=\widetilde{\Pi_{\Psi}'}(g)m(g)^{-1} t'(g)$.  Let us define $t''(g)=m(g)^{-1} t'(g)$. If $g\in B\cap \Gamma_1'$, $t''(g)\in \mu_4$. If $g\in N^-\cap \Gamma_1'$, let $h=\omega^{-1}g\omega$. Then:
$$\widetilde{\Pi_{\Psi}'}(g)=\widetilde{c_{PR, \mathcal{X}^{\ast}}}(\omega,h)\widetilde{c_{PR, \mathcal{X}^{\ast}}}(\omega h, \omega^{-1})\widetilde{\Pi_{\Psi}'}(\omega)\widetilde{\Pi_{\Psi}'}(h)\widetilde{\Pi_{\Psi}'}(\omega^{-1}) $$
$$=\widetilde{c_{PR, \mathcal{X}^{\ast}}}(\omega,h)\widetilde{c_{PR, \mathcal{X}^{\ast}}}(\omega h, \omega^{-1})m(\omega)^{-1} m(h)^{-1}m(\omega^{-1})^{-1}\Pi_{\Psi}'(\omega)\Pi_{\Psi}'(h)\Pi_{\Psi}'(\omega^{-1}).$$
Hence: $$t''(g)^{-1}=\widetilde{c_{PR, \mathcal{X}^{\ast}}}(\omega,h)\widetilde{c_{PR, \mathcal{X}^{\ast}}}(\omega h, \omega^{-1})m(\omega)^{-1} m(h)^{-1}m(\omega^{-1})^{-1}.$$
Note that $ m(h)^{-1}=1$, and $m(\omega)^{-1}m(\omega^{-1})^{-1}\in \mu_4$. Hence, $t''(g)^{-1}\in\mu_4$.

Assume  $g=s(\dot{g})g_1$, for some $g_1\in \Gamma_1'$.  Let us define a function as follows:
$$t'': \Sp(  \mathscr{W}_{\mathfrak{O}})\longrightarrow \mu_8; g \longmapsto \left\{ \begin{array}{ccl} \widetilde{c_{PR, \mathcal{X}^{\ast}}}(s(\dot{g}),g_1)t''(g_1) & &g=s(\dot{g})g_1 \notin \Gamma_1',\\
t''(g_1) & & g\in \Gamma_1'. \end{array} \right.$$  Let us modify the action of $ \Sp(\mathscr{W}_{\mathfrak{O}})$  on $V_{\psi}$ by defining
$$\Pi^{''\sim}_{\Psi}(g)f=\widetilde{\Pi_{\Psi}'}(g) t''(g)f, \qquad f\in V_{\psi}.$$
For $g=s(\dot{g})g_1\in  \Sp(\mathscr{W}_{\mathfrak{O}})$, if $s(\dot{g})=1$, then:
$$\Pi^{''\sim}_{\Psi}(g)f=\widetilde{\Pi_{\Psi}'}(g) t''(g)f=\Pi^{'\sim}_{\Psi}(g)f=f.$$
If $s(\dot{g})\neq 1$, then: $$\Pi^{''\sim}_{\Psi}(g)f=\widetilde{\Pi_{\Psi}'}(g) t''(g)f= \widetilde{c_{PR, \mathcal{X}^{\ast}}}(s(\dot{g}),g_1)^{-1} \widetilde{\Pi_{\Psi}'}(s(\dot{g}))\widetilde{\Pi_{\Psi}'}(g_1)t''(g)f=\widetilde{\Pi_{\Psi}'}(s(\dot{g}))\Pi^{'\sim}_{\Psi}(g_1)f=\widetilde{\Pi_{\Psi}'}(s(\dot{g}))f.$$
Let $c_{\Pi^{''\sim}_{\Psi}}(-,-)$ be  the $2$-cocycle attached to $\Pi^{'\sim}_{\Psi}$.
\begin{lemma}
 $c_{\Pi^{''\sim}_{\Psi}}(g_1,g_2)= \widetilde{c_{PR, \mathcal{X}^{\ast}}}(s(\dot{g}_1),s(\dot{g}_2))t''(s(\dot{g}_1)s(\dot{g}_2))^{-1} \in \mu_4$, for $g_i\in \Sp(  \mathscr{W}_{\mathfrak{O}})$.
\end{lemma}
\begin{proof}
For  $f\in V_{\psi}$,
$$\Pi^{''\sim}_{\Psi}(g_1)\Pi^{''\sim}_{\Psi}(g_2)f=\widetilde{\Pi_{\Psi}'}(s(\dot{g}_1))\widetilde{\Pi_{\Psi}'}(s(\dot{g}_2))f=\widetilde{c_{PR, \mathcal{X}^{\ast}}}(s(\dot{g}_1),s(\dot{g}_2))\widetilde{\Pi_{\Psi}'}(s(\dot{g}_1)s(\dot{g}_2))f$$
$$= \widetilde{c_{PR, \mathcal{X}^{\ast}}}(s(\dot{g}_1),s(\dot{g}_2))\widetilde{\Pi_{\Psi}'}(s(\dot{g}_1)s(\dot{g}_2))\widetilde{\Pi_{\Psi}'}(s(\dot{g}_1\dot{g}_2))^{-1}  \Pi^{''\sim}_{\Psi}(g_1g_2)f.$$

Hence $c_{\Pi^{''\sim}_{\Psi}}(g_1,g_2)= \widetilde{c_{PR, \mathcal{X}^{\ast}}}(s(\dot{g}_1),s(\dot{g}_2))\widetilde{\Pi_{\Psi}'}(s(\dot{g}_1)s(\dot{g}_2))\widetilde{\Pi_{\Psi}'}(s(\dot{g}_1\dot{g}_2))^{-1}  $. Note that $\Pi^{''\sim}_{\Psi}(s(\dot{g}_1)s(\dot{g}_2))=\Pi^{''\sim}_{\Psi}(s(\dot{g}_1\dot{g}_2))$. Hence
$$\widetilde{\Pi_{\Psi}'} (s(\dot{g}_1)s(\dot{g}_2))=\Pi^{''\sim}_{\Psi}(s(\dot{g}_1)s(\dot{g}_2)) t''(s(\dot{g}_1)s(\dot{g}_2))^{-1}$$
$$=\Pi^{''\sim}_{\Psi}(s(\dot{g}_1\dot{g}_2))t''(s(\dot{g}_1)s(\dot{g}_2))^{-1}=\widetilde{\Pi_{\Psi}'}(s(\dot{g}_1\dot{g}_2))t''(s(\dot{g}_1)s(\dot{g}_2))^{-1}.$$
So the result holds.
\end{proof}
\begin{question}
Is it also possible to verify that the $2$-cocycle can be valued on  $\mu_4$ through $c_{PR, \mathcal{X}^{\ast}}$, by selecting the appropriate section map from $\Sp(W)$ to $\Sp(\mathscr{W}_{\mathfrak{O}})$.
\end{question}
\subsection{Extending to an affine symplectic similitude  group}
Follow the notations of Section \ref{ASG}. Let $\mathfrak{u}$ denote the subgroup  of units  of $\mathfrak{O}$, and $\mathfrak{u}_1=1+\mathfrak{m}$.  Let $\GSp(\mathscr{W}_{ \mathfrak{O}})=\{ g\in \GSp(\mathscr{W}), \lambda_g \in \mathfrak{u}\}$.  Let $GP's(\mathscr{W}_{ \mathfrak{O}})$ be the set  of elements $(g, q')$ of $GP's(\mathscr{W})$ such that $g\in \GSp(\mathscr{W}_{ \mathfrak{O}})$, and $q':
\mathscr{W}_{ \mathfrak{O}} \longrightarrow \mathfrak{O}$, $q':
\mathscr{W}_{ \mathfrak{m}} \longrightarrow \mathfrak{m}^2$, $q':
\frac{\mathscr{W}_{ \mathfrak{O}}}{\mathscr{W}_{ \mathfrak{m}}} \longrightarrow \frac{ \mathfrak{O}}{\mathfrak{m}^2}$ is a quadratic function.
Let $G\Gamma_1=\ker( \GSp(\mathscr{W}_{ \mathfrak{O}}) \longrightarrow \GSp(\mathscr{W}_{ \mathfrak{O}}/\mathscr{W}_{ \mathfrak{m}}))$, and $GP's(\mathscr{W}_{ \mathfrak{O}})_1=\{ (g, q') \in GP's(\mathscr{W}_{ \mathfrak{O}}) \mid g \in G\Gamma_1, q': \mathscr{W}_{ \mathfrak{O}} \longrightarrow \mathfrak{m}^2\}$. Recall the group homomorphism $\alpha': \GSp(\mathscr{W}_{ \mathfrak{O}}) \longrightarrow  GP's(\mathscr{W}); g\longmapsto (g, q_g')$, where $q_g'(x+x^{\ast})= \frac{1}{2}B'(
xa, xb)+\frac{1}{2}B'( x^{\ast}c, x^{\ast}d)+B'(x^{\ast}c, xb)$, for $w\in\mathscr{W}_{ \mathfrak{O}}$,  $g=\begin{pmatrix} a& b\\c & d\end{pmatrix}\in \GSp(\mathscr{W}_{ \mathfrak{O}} )$. Hence the image lies in $GP's(\mathscr{W}_{ \mathfrak{O}})$.
\begin{lemma}
\begin{itemize}
\item[(1)] There exists a split exact sequence: $1\longrightarrow \mathscr{W}_{ \mathfrak{O}}^{'\vee} \longrightarrow GP's(\mathscr{W}_{ \mathfrak{O}}) \longrightarrow \GSp(\mathscr{W}_{ \mathfrak{O}}) \longrightarrow 1$.
\item[(2)] There exists an exact sequence: $1\longrightarrow \mathscr{W}_{ \mathfrak{O}}^{''\vee} \longrightarrow GP's(\mathscr{W}_{ \mathfrak{O}})_1 \longrightarrow G\Gamma_1 \longrightarrow 1$.
    \end{itemize}
\end{lemma}
\begin{proof}
1) One can see that the image of $\alpha'$ belongs to $GP's(\mathscr{W}_{ \mathfrak{O}})$. Hence the exact sequence is split. \\
2) For $g\in G\Gamma_1$, let us write $g=1+2g_1$, for some $g_1\in\End( \mathscr{W}_{ \mathfrak{O}})$. Then $g^{-1}=1-2g_1 $ mod $4 \End( \mathscr{W}_{ \mathfrak{O}})$. So for any $w=x+x^{\ast}\in \mathscr{W}_{\mathfrak{O}}$,
$$(\lambda_g-1) \frac{1}{2}B'(w,w)=\frac{1}{2}\langle xg, x^{\ast}g\rangle'-\frac{1}{2}\langle x, x^{\ast}\rangle'=\langle xg_1, x^{\ast}2g_1\rangle'+\langle x2g_1,x^{\ast}\rangle+\langle x,x^{\ast}2g_1\rangle$$
$$=\langle xg_1, x^{\ast}2g_1\rangle'+\langle xg,x^{\ast}\rangle+\langle x,x^{\ast}g\rangle=\langle xg_1, x^{\ast}2g_1\rangle'+\langle x,x^{\ast}g^{-1}+x^{\ast}g\rangle$$
$$= \langle x,x^{\ast}\rangle' \textrm{ mod } \mathfrak{m}^2 = B'(w,w)\textrm{ mod } \mathfrak{m}^2 $$

$$q'_g(w)=\frac{1}{2}B'(wg,wg)-\lambda_g\frac{1}{2}B'(w,w)=\frac{1}{2}B'(2wg_1,2wg_1)+\frac{1}{2}B'(2wg_1,w)+\frac{1}{2}B'(w, 2wg_1)+(1-\lambda_g) \frac{1}{2}B'(w,w)$$
$$=2B'(wg_1,wg_1)+\langle wg_1,wg\rangle'+B'(w, 2wg_1)+(1-\lambda_g) \frac{1}{2}B'(w,w).$$
Note that $2B'(wg_1,wg_1)+B'(w, 2wg_1) \in \mathfrak{m}^2$. On the other hand, let us define $f(w)=-\langle wg_1,wg\rangle'+ B'(w,w)=\langle w,wg\rangle+B'(w,w)\in \mathfrak{m}$. Then $f|_{\mathscr{W}_{ \mathfrak{m}}}\in \mathfrak{m}^2$, and $f(w_1+w_2)=f(w_1)+f(w_2) \textrm{ modulo } \mathfrak{m}^2$. Hence the result holds, similar to the second part of the  proof of  Lemma \ref{sec}.
\end{proof}

 For our purpose, let us consider some subgroups of $\GSp(\mathscr{W}_{ \mathfrak{O}})$.  Recall that $\mathfrak{u}$ is  the subgroup  of units  of $\mathfrak{O}$, and $\mathfrak{u}_1=1+\mathfrak{m}$.
\begin{lemma}
$[\mathfrak{u}^2:\mathfrak{u}_1^2]=[\mathfrak{u}:\mathfrak{u}_1]=|F^{\times}|$.
\end{lemma}
\begin{proof}
We shall follow the proofs from \cite{Gers}. Let $\varphi: \mathfrak{u} \longrightarrow \mathfrak{u}^2; x \longrightarrow x^2$. Then $\ker \varphi=\{\pm 1\}$, and $-1=1+(-2)\in 1+\mathfrak{m}$. By \cite[Lmm.3.47]{Gers}, $ [\mathfrak{u}:\mathfrak{u}_1]=[\varphi(\mathfrak{u}): \varphi(\mathfrak{u}_1)][\ker \varphi: \ker\varphi\cap \mathfrak{u}_1]=[\mathfrak{u}^2: \mathfrak{u}_1^2]$.
\end{proof}
Let $ \mathfrak{u}^2\Sp(\mathscr{W}_{ \mathfrak{O}})=\{ g\in \GSp(\mathscr{W}_{ \mathfrak{O}})\mid \lambda_g\in \mathfrak{u}^2\}$.
Let $\mathfrak{u}^2P's(\mathscr{W}_{ \mathfrak{O}})=\{ (g,q)\in GP's(\mathscr{W}_{ \mathfrak{O}})\mid  g\in \mathfrak{u}^2\Sp(\mathscr{W}_{ \mathfrak{O}})\}$, $\mathfrak{u}^2_1\Gamma_1=G\Gamma_1 \cap  \mathfrak{u}^2\Sp(\mathscr{W}_{ \mathfrak{O}})$, and $\mathfrak{u}^2_1P's(\mathscr{W}_{ \mathfrak{O}})_1=GP's(\mathscr{W}_{ \mathfrak{O}})_1 \cap  \mathfrak{u}^2P's(\mathscr{W}_{ \mathfrak{O}})$.  Note that for $(g,q)\in \mathfrak{u}^2_1\Gamma_1$, $\lambda_g\in \mathfrak{u}^2_1$.

 Follow the notations of Section \ref{ASG}. Then we have the following diagram analogue of $(\ref{qeq})$.
\begin{equation}\label{qeq2}
\begin{CD}
@. 1@. 1 @.1 @. \\
@.@VVV @VVV @VVV @.\\
1 @>>> \mathscr{W}_{ \mathfrak{O}}^{''\vee} @>>> \mathfrak{u}_1^2P's(\mathscr{W}_{ \mathfrak{O}})_1@>>> \mathfrak{u}^2_1\Gamma_1 @>>> 1\\
@.@VVV @VVV @VVV @.\\
1 @>>> \mathscr{W}_{ \mathfrak{O}}^{'\vee} @>>> \mathfrak{u}^2P's(\mathscr{W}_{ \mathfrak{O}})@>>> \mathfrak{u}^2\Sp(\mathscr{W}_{ \mathfrak{O}}) @>>> 1\\
@.@VVV @VVV @VVV @.\\
1@>>>W^{\vee} @>>> AGSp(W) @>>>\GSp(W) @>>> 1\\
@.@VVV @VVV @VVV @.\\
@. 1@. 1 @.1 @.
\end{CD}
\end{equation}

\subsection{$GMp(\mathscr{W})$ case}
By Section \ref{GLGU},  $ AGSp(W) \simeq  ASp(W) \rtimes F^{\times}$, and  there exists an exact sequence: $1\longrightarrow \mu_8 \longrightarrow AMp(W)\longrightarrow ASp(W) \longrightarrow 1$. Note that $P's(\mathscr{W}_{ \mathfrak{O}})/P's(\mathscr{W}_{ \mathfrak{O}})_1 \simeq  ASp(W)$, $\Sp(\mathscr{W}_{ \mathfrak{O}})/\Gamma'_1\simeq \Sp(W)$.
The $2$-cocycle attached to $AMp(W)$ is given by the $2$-cocycle $c_{\Pi^{'\sim}_{\Psi}}(-,-)$, which satisfies some  properties  as shown in  Lemma \ref{reduction}. Moreover, $c_{\Pi^{'\sim}_{\Psi}}|_{\Sp(\mathscr{W}_{ \mathfrak{O}}) \times \Gamma'_1}=1=c_{\Pi^{'\sim}_{\Psi}}|_{\Gamma'_1 \times \Sp(\mathscr{W}_{ \mathfrak{O}})}$.
In \cite{Ba}, Barthel can extend the  $2$-cocycle $c_{\Pi^{'\sim}_{\Psi}}(-,-)$ from $\Sp( \mathscr{W}_{ \mathfrak{O}})$ to $ \GSp(\mathscr{W}_{ \mathfrak{O}})$. For simplicity, we also write  $c_{\Pi^{'\sim}_{\Psi}}(-,-)$. By \cite[Sects.1.2.1-1.2.3]{Ba},
$c_{\Pi^{'\sim}_{\Psi}}|_{\mathfrak{u}^2\Sp(\mathscr{W}_{ \mathfrak{O}}) \times \mathfrak{u}_1^2\Gamma_1'}=1=c_{\Pi^{'\sim}_{\Psi}}|_{\mathfrak{u}_1^2\Gamma_1' \times \mathfrak{u}^2\Sp(\mathscr{W}_{ \mathfrak{O}}) }$.  By using Barthel's $c_{\Pi^{'\sim}_{\Psi}}$, we also obtain  an exact sequence: $1\longrightarrow \mu_8 \longrightarrow AGMp(W)\longrightarrow AGSp(W) \longrightarrow 1$, and $AGMp(W) \supseteq AMp(W)$.
By \cite[Sects.1.2.1-1.2.3]{Ba}, $c_{\Pi^{'\sim}_{\Psi}}|_{\mathfrak{u}^2\Sp(\mathscr{W}_{ \mathfrak{O}}) \times \mathfrak{u}^2}=1=c_{\Pi^{'\sim}_{\Psi}}|_{ \mathfrak{u}^2\times \mathfrak{u}^2\Sp(\mathscr{W}_{ \mathfrak{O}}) }$,
so $AGMp(W) \simeq AMp(W) \rtimes  F^{\times}$. Hence the Weil representation of $AMp(W)$ can extend to $AGMp(W) $ by adding some  character on $F^{\times}$.

  \section{Appendix: Genestier-Lysenko and Gurevich-Hadani's results}\label{GLGU}
  Let $\Q_2$ be the $2$-adic complete field from the rational number field. Let $K/\Q_2$ be an unramified field extension of degree $d$. Let $\mathfrak{O}_K$ or $\mathfrak{O}$ denote  the ring of integers of $K$. Then $2$ is one uniformizer of $K$. Let  $\mathfrak{m}^n=2^n \mathfrak{O}$. Let $F=  \mathfrak{O}_K/ \mathfrak{m}_K$. Let  $R=  \mathfrak{O}_K/ \mathfrak{m}^2_K $.   Let $(\widetilde{W}, \langle, \rangle_{\widetilde{W}})$ be a free symplectic module over $R$ with a symplectic basis $\{e_1, \cdots, e_m; e_1^{\ast}, \cdots, e_m^{\ast}\}$. Let $\widetilde{X}=\Span_R\{e_1, \cdots, e_m\}$, $\widetilde{X^{\ast}}=\Span_R\{e_1^{\ast}, \cdots, e_m^{\ast}\}$.  For $w=x+x^{\ast}$, $w'=y+y^{\ast}$, $x=\sum_{i=1}^m e_ix_i$, $x^{\ast}=\sum_{i=1}^m  e_i^{\ast}x_i^{\ast}$, $y=\sum_{i=1}^m e_iy_i$, $y^{\ast}=\sum_{i=1}^m  e_i^{\ast}y_i^{\ast}$,
  $$\langle w,w' \rangle_{\widetilde{W}}= \sum_{i=1}^m x_i y_i^{\ast}-\sum_{j=1}^m y_j x_j^{\ast}.$$
Let $\widetilde{\beta}(-,-): \widetilde{W} \times \widetilde{W}\longrightarrow R; (x+x^{\ast}, y+y^{\ast}) \longmapsto \langle x,y^{\ast} \rangle_{\widetilde{W}}$.  Then $\langle w,w' \rangle_{\widetilde{W}}=\widetilde{\beta}(w, w')-\widetilde{\beta}(w',w)$.

Let $W=\widetilde{W}/(2)$, which is a  vector space over $F$ of dimension $2m$.   Following \cite{GeLy},\cite{GuHa},  let us define a non-degenerate form $\langle-,-\rangle_W$ on $W$ as follows:
$$\langle, \rangle_W: W \times W \longrightarrow R; (w, w') \longmapsto   2\langle w, w'\rangle_{\widetilde{W}}.$$
It can be checked that it is well-defined. Let $\beta: W\times W \longrightarrow R; (w, w') \longmapsto 2\widetilde{\beta}(w,w')$.  Then $\langle w,w' \rangle_{W}=\beta(w, w')-\beta(w',w)$.
Let $H_{\beta}(W)=W\times R$, denote   the corresponding Heisenberg group, defined as follows:
 $$(w, t)+(w',t')=(w+w', t+t'+\beta(w,w')).$$
 Let $\psi: R\longrightarrow \C^{\times}$ be a  faithful character.
 \begin{theorem}[Stone-von Neumann]
 There exists a unique (up to
isomorphism) irreducible Heisenberg representation of $H_{\beta}(W)$ with the central character $\psi$.
 \end{theorem}
 \begin{proof}
 See \cite[Thm.1.2]{GuHa}.
 \end{proof}
 \subsubsection{Affine symplectic group}
 Following \cite{GeLy},\cite{GuHa}, for $g\in \Sp(W)$,  we let $\Sigma_g$  be the set  of   ``quadratic functions'' $q$ (see \cite{GeLy}, \cite{GuHa} for details) from $W$ to $R$ such that
 \begin{equation}\label{equiv1}
 q(w_1+w_2)-q(w_1)-q (w_2)= \beta(w_1g, w_2g)-\beta(w_1, w_2)
 \end{equation}
 Another  function  $q'\in \Sigma_g$ iff  $q-q'$ is an additive quadratic function.  The affine symplectic group $ASp(W)=\{ ( g, q)\mid  g\in \Sp(W), q\in \Sigma_g \}$, with the group law $$( g, q) ( g', q')=(gg', q^{''}),$$
 where $q^{''}(w)=q(w)+q'(wg)$, for $w\in W$.
 \begin{lemma}
For any $g\in \Sp(W)$, $\Sigma_g\neq \emptyset$. Moreover,  there exists a non-split group extension: $1\longrightarrow W^{\vee} \longrightarrow ASp(W) \longrightarrow \Sp(W) \longrightarrow 1$, where $W^{\vee}=\Hom(W, R)$.
 \end{lemma}
 \begin{proof}
 See \cite[Section 2.2]{GeLy}, \cite[Section 1.3]{GuHa}.
 \end{proof}

  For $g\in \GSp(W)$,  let $\Sigma_g$ also  denote to be  the set  of quadratic functions $q$ from $W$ to $R$ such that
 \begin{equation}\label{equiv}
 q(w_1+w_2)-q(w_1)-q (w_2)= \beta(w_1g, w_2g)-\lambda_g\beta(w_1, w_2)
 \end{equation}
 where $\lambda: \GSp(W) \longrightarrow F^{\times}$,   the similitude factor.  The affine similitude  symplectic group $AGSp(W)=\{ ( g, q)\mid  g\in \Sp(W), q\in \Sigma_g \}$, with the group law $$( g, q) ( g', q')=(gg',q''),$$
 where    $q^{''}(w)=\lambda_{g'}q(w)+q'(wg)$,  for $w\in W$.

 \begin{lemma}
\begin{itemize}
\item[(1)]  There exists an exact sequence: $1\longrightarrow W^{\vee} \longrightarrow AGSp(W) \longrightarrow \GSp(W) \longrightarrow 1$.
\item[(2)] There exists a split exact sequence: $1\longrightarrow ASp(W)\longrightarrow AGSp(W) \longrightarrow F^{\times} \longrightarrow 1$.
    \end{itemize}
\end{lemma}
\begin{proof}
1) Follow the proof of  \cite[Lmm.1.1]{GuHa}. For any $g\in \GSp(W)$, let $\widetilde{g}$ be one of its lifting in $\GSp(\widetilde{W})$.
  Note that $\langle w,w'\rangle= 2 \langle \widetilde{w}, \widetilde{w'}\rangle_{\widetilde{W}} $, $\beta(w,w')=2\widetilde{\beta}(\widetilde{w}, \widetilde{w'})$,  for any lifting $\widetilde{w}$ of $w$, and $ \widetilde{w'}$ of $w'$.
  Then: $$\lambda_g\langle w,w'\rangle=\langle wg,w'g\rangle =2 \langle \widetilde{w}\widetilde{g}, \widetilde{w'}\widetilde{g}\rangle_{\widetilde{W}} =2 \lambda_{\widetilde{g}}\langle \widetilde{w}, \widetilde{w'}\rangle_{\widetilde{W}},$$
  $$\beta(wg,w'g)=2\widetilde{\beta}(\widetilde{w}\widetilde{g}, \widetilde{w'}\widetilde{g}), \quad \lambda_g\beta(w,w')= 2  \lambda_{\widetilde{g}}\widetilde{\beta}(\widetilde{w}, \widetilde{w'}).$$
Let $q_{\widetilde{g}}(\widetilde{w})=\widetilde{\beta}(\widetilde{w}\widetilde{g}, \widetilde{w}\widetilde{g})-\lambda_{\widetilde{g}}\widetilde{\beta}(\widetilde{w}, \widetilde{w})$, for $\widetilde{w}\in \widetilde{W}$. Then:
$$q_{\widetilde{g}}(\widetilde{w}+2\widetilde{w}')=q_{\widetilde{g}}(\widetilde{w})+2\widetilde{\beta}(\widetilde{w}\widetilde{g}, \widetilde{w}'\widetilde{g})+2\widetilde{\beta}(\widetilde{w}'\widetilde{g}, \widetilde{w}\widetilde{g})-2\lambda_{\widetilde{g}}\widetilde{\beta}(\widetilde{w}', \widetilde{w})-2\lambda_{\widetilde{g}}\widetilde{\beta}(\widetilde{w}, \widetilde{w}')$$
$$=q_{\widetilde{g}}(\widetilde{w})+\beta(wg, w'g)+\beta(w'g, wg)-\lambda_{g}\beta(w',w)-\lambda_{g}\beta(w, w')$$
$$=q_{\widetilde{g}}(\widetilde{w})+\langle w'g, wg\rangle-\lambda_{g}\langle w', w\rangle=q_{\widetilde{g}}(\widetilde{w}).$$
So $q_{\widetilde{g}}$ defines a function from $W$ to $R$. Moreover, $q_{\widetilde{g}}$ is a quadratic function. (cf.\cite{GeLy}). Let us check the equality (\ref{equiv}).
$$q_{\widetilde{g}}(\widetilde{w}+\widetilde{w}')-q_{\widetilde{g}}(\widetilde{w}')-q_{\widetilde{g}}(\widetilde{w}')$$
$$=\widetilde{\beta}((\widetilde{w}+\widetilde{w}')\widetilde{g}, (\widetilde{w}+\widetilde{w}')\widetilde{g})-\lambda_{\widetilde{g}}\widetilde{\beta}(\widetilde{w}+\widetilde{w}', \widetilde{w}+\widetilde{w}')-\widetilde{\beta}(\widetilde{w}\widetilde{g}, \widetilde{w}\widetilde{g})+\lambda_{\widetilde{g}}\widetilde{\beta}(\widetilde{w}, \widetilde{w})-\widetilde{\beta}(\widetilde{w}'\widetilde{g}, \widetilde{w}'\widetilde{g})+\lambda_{\widetilde{g}}\widetilde{\beta}(\widetilde{w}', \widetilde{w}')$$
$$=\widetilde{\beta}(\widetilde{w}'\widetilde{g}, \widetilde{w}\widetilde{g})+\widetilde{\beta}(\widetilde{w}\widetilde{g}, \widetilde{w}'\widetilde{g})-\lambda_{\widetilde{g}}\widetilde{\beta}(\widetilde{w}', \widetilde{w})-\lambda_{\widetilde{g}}\widetilde{\beta}(\widetilde{w}, \widetilde{w}')$$
$$=\langle\widetilde{w}'\widetilde{g}, \widetilde{w}\widetilde{g}\rangle+2\widetilde{\beta}(\widetilde{w}\widetilde{g}, \widetilde{w}'\widetilde{g})-\lambda_{\widetilde{g}}\langle\widetilde{w}', \widetilde{w}\rangle-2\lambda_{\widetilde{g}}\widetilde{\beta}(\widetilde{w}, \widetilde{w}')$$
$$=2\widetilde{\beta}(\widetilde{w}\widetilde{g}, \widetilde{w}'\widetilde{g})-2\lambda_{\widetilde{g}}\widetilde{\beta}(\widetilde{w}, \widetilde{w}')$$
$$=\beta(wg, w'g)-\lambda_{g}\beta(w, w').$$
So there exists a transversal map $\alpha: \GSp(W) \longrightarrow AGSp(W); g \longrightarrow (g, q_{\widetilde{g}})$.\\
2) $\GSp(W)\simeq \Sp(W) \rtimes F^{\times}$, where $h: F^{\times } \longrightarrow \GSp(W); t \longmapsto h_t= \begin{pmatrix} 1&0\\0 & t\end{pmatrix}$. Then $\alpha(h_t)=(h_t, 0)$. So $\alpha|_{F^{\times}}$ is a group homomorphism.
\end{proof}

  The affine symplectic group  can act on $H_{\beta}(W)$ in the following way: $(w, t)( g, q) =(wg, t+q(w))$. This action will extend  the above Heisenberg reprentation  to a projective representation of $ ASp(W)\ltimes H_{\beta}(W) $.  Let $\mu_4=\{e^{\frac{2k\pi i}{4}}\mid 1\leq k\leq 4\}\subseteq \C$.
\begin{theorem}
 There exists a non-trivial central   extension: $1\longrightarrow \mu_4 \longrightarrow AMp(W)\longrightarrow ASp(W) \longrightarrow 1$,  such that $\pi_{\psi}$ can extend to be a representation of $AMp(W)\ltimes H_{\beta}(W)$.
 \end{theorem}
 \begin{proof}
 See \cite[Thm.1.3]{GuHa}, \cite{GeLy}.
 \end{proof}
 As usual, $\pi_{\psi}|_{AMp(W)}$ is called the Weil representation of $AMp(W)$. Recall the above $R$-module $\widetilde{W}$. Let $Mp(\widetilde{W})$ be a non-trivial central extension of $\Sp(\widetilde{W})$ by $\mu_2$.
 \begin{theorem}
 There exists a group homomorphism from $Mp(\widetilde{W})$ to $AMp(W)$.
 \end{theorem}
\begin{proof}
 See \cite[Thm.1.4]{GuHa}, \cite[Prop.1]{GeLy}.
 \end{proof}
\labelwidth=4em
\addtolength\leftskip{25pt}
\setlength\labelsep{0pt}

\end{document}